\documentclass[12pt]{amsart}
\usepackage{amsmath, amssymb}

\numberwithin{equation}{section}
\theoremstyle{plain}
\newtheorem{theorem}{Theorem}[section]
\newtheorem{theoremintro}{Theorem}

\newtheorem{lemma}[theorem]{Lemma}
\newtheorem{prop}[theorem]{Proposition}
\newtheorem{corollary}[theorem]{Corollary}

\theoremstyle{definition}
\newtheorem{definition}[theorem]{Definition}
\newtheorem{remark}[theorem]{Remark}

\usepackage{amsmath}
\usepackage{amsfonts}
\usepackage{amssymb}
\usepackage{amscd}
\usepackage{color}
\usepackage{xcolor}
\usepackage{hyperref}
\usepackage{mathrsfs}
\usepackage{eucal}
\usepackage{upgreek}
\usepackage[makeroom]{cancel}
\usepackage[normalem]{ulem}
\usepackage{array}
\usepackage{verbatim}
\usepackage{mathtools}
\usepackage{stmaryrd}

\usepackage[inline]{enumitem}

\usepackage{xy}
\xyoption{all}
\usepackage{tikz-cd}

\newcommand{\nc}{\newcommand}

\newcommand{\Z}{\mathbb Z}
\newcommand{\Q}{\mathbb Q}

\renewcommand{\L}{\mathbb L}
\renewcommand{\P}{\mathbb P}

\newcommand{\kk}{\Bbbk}

\newcommand{\xra}{\xrightarrow}
\renewcommand{\le}{\leqslant}
\renewcommand{\ge}{\geqslant}

\newcommand{\bul}{\bullet}

\newcommand{\modd}{\mathrm{mod{-}}}
\newcommand{\moddo}{\mathrm{mod}_0{-}}

\DeclareMathOperator{\Hom}{\textup{Hom}}

\DeclareMathOperator{\Ext}{\textup{Ext}}

\DeclareMathOperator{\Pic}{\mathrm{Pic}}
\DeclareMathOperator{\End}{\mathrm{End}}

\DeclareMathOperator{\Spec}{\mathrm{Spec}}
\DeclareMathOperator{\har}{\mathrm{char}} 
\DeclareMathOperator{\Ob}{\mathrm{Ob}}

\DeclareMathOperator{\im}{\mathrm{im}}
\DeclareMathOperator{\Aut}{\mathrm{Aut}}
\DeclareMathOperator{\rank}{\mathrm{rank}}

\DeclareMathOperator{\coh}{\mathrm{coh}}

\DeclareMathOperator{\Db}{D^{\mathrm b}}

\DeclareMathOperator{\id}{\mathrm{id}}

\DeclareMathOperator{\Gal}{\mathrm{Gal}}

\DeclareMathOperator{\Iso}{\mathrm{Iso}}
\DeclareMathOperator{\GIso}{\mathrm{G{-}Iso}}
\DeclareMathOperator{\ev}{\mathrm{ev}}
\DeclareMathOperator{\tors}{\mathrm{tors}}
\DeclareMathOperator{\dP4}{\mathrm{dP4}}
\DeclareMathOperator{\GdP4}{\mathrm{G{-}dP4}}
\DeclareMathOperator{\TrCat}{\mathrm{TrCat}}
\DeclareMathOperator{\GTrCat}{\mathrm{G{-}TrCat}}
\DeclareMathOperator{\PGL}{\mathrm{PGL}}
\DeclareMathOperator{\numerical}{\mathrm{num}}

\nc{\cA}{{\mathcal{A}}}
\nc{\cB}{{\mathcal{B}}}
\nc{\cC}{{\mathcal{C}}}
\nc{\cD}{{\mathcal{D}}}
\nc{\cE}{{\mathcal{E}}}
\nc{\cF}{{\mathcal{F}}}
\nc{\cG}{{\mathcal{G}}}
\nc{\cH}{{\mathcal{H}}}
\nc{\cI}{{\mathcal{I}}}
\nc{\cJ}{{\mathcal{J}}}
\nc{\cK}{{\mathcal{K}}}
\nc{\cL}{{\mathcal{L}}}
\nc{\cM}{{\mathcal{M}}}
\nc{\cN}{{\mathcal{N}}}
\nc{\cO}{{\mathcal{O}}}
\nc{\cP}{{\mathcal{P}}}
\nc{\cQ}{{\mathcal{Q}}}
\nc{\cR}{{\mathcal{R}}}
\nc{\cS}{{\mathcal{S}}}
\nc{\cT}{{\mathcal{T}}}
\nc{\cU}{{\mathcal{U}}}
\nc{\cV}{{\mathcal{V}}}
\nc{\cW}{{\mathcal{W}}}
\nc{\cX}{{\mathcal{X}}}
\nc{\cY}{{\mathcal{Y}}}
\nc{\cZ}{{\mathcal{Z}}}

\nc{\rc}{{\mathrm{c}}}
\nc{\rd}{{\mathrm{d}}}
\nc{\rf}{{\mathrm{f}}}
\nc{\rh}{{\mathrm{h}}}
\nc{\rrm}{{\mathrm{m}}}
\nc{\rs}{{\mathrm{s}}}
\nc{\rch}{{\mathrm{ch}}}
\nc{\rtd}{{\mathrm{td}}}

\nc{\rA}{{\mathrm{A}}}
\nc{\rB}{{\mathrm{B}}}
\nc{\rC}{{\mathrm{C}}}
\nc{\rD}{{\mathrm{D}}}
\nc{\rE}{{\mathrm{E}}}
\nc{\rF}{{\mathrm{F}}}
\nc{\rG}{{\mathrm{G}}}
\nc{\rH}{{\mathrm{H}}}
\nc{\rI}{{\mathrm{I}}}
\nc{\rJ}{{\mathrm{J}}}
\nc{\rK}{{\mathrm{K}}}
\nc{\rL}{{\mathrm{L}}}
\nc{\rM}{{\mathrm{M}}}
\nc{\rN}{{\mathrm{N}}}
\nc{\rO}{{\mathrm{O}}}
\nc{\rP}{{\mathrm{P}}}
\nc{\rQ}{{\mathrm{Q}}}
\nc{\rR}{{\mathrm{R}}}
\nc{\rS}{{\mathrm{S}}}
\nc{\rT}{{\mathrm{T}}}
\nc{\rU}{{\mathrm{U}}}
\nc{\rV}{{\mathrm{V}}}
\nc{\rW}{{\mathrm{W}}}
\nc{\rX}{{\mathrm{X}}}
\nc{\rY}{{\mathrm{Y}}}
\nc{\rZ}{{\mathrm{Z}}}

\nc{\bA}{{\mathbf{A}}}
\nc{\bB}{{\mathbf{B}}}
\nc{\bC}{{\mathbf{C}}}
\nc{\bD}{{\mathbf{D}}}
\nc{\bE}{{\mathbf{E}}}
\nc{\bF}{{\mathbf{F}}}
\nc{\bG}{{\mathbf{G}}}
\nc{\bH}{{\mathbf{H}}}
\nc{\bI}{{\mathbf{I}}}
\nc{\bJ}{{\mathbf{J}}}
\nc{\bK}{{\mathbf{K}}}
\nc{\bL}{{\mathbf{L}}}
\nc{\bM}{{\mathbf{M}}}
\nc{\bN}{{\mathbf{N}}}
\nc{\bO}{{\mathbf{O}}}
\nc{\bP}{{\mathbf{P}}}
\nc{\bQ}{{\mathbf{Q}}}
\nc{\bR}{{\mathbf{R}}}
\nc{\bS}{{\mathbf{S}}}
\nc{\bT}{{\mathbf{T}}}
\nc{\bU}{{\mathbf{U}}}
\nc{\bV}{{\mathbf{V}}}
\nc{\bW}{{\mathbf{W}}}
\nc{\bX}{{\mathbf{X}}}
\nc{\bY}{{\mathbf{Y}}}
\nc{\bZ}{{\mathbf{Z}}}

\nc{\ba}{{\mathbf{a}}}
\nc{\bb}{{\mathbf{b}}}
\nc{\bc}{{\mathbf{c}}}
\nc{\bd}{{\mathbf{d}}}
\nc{\be}{{\mathbf{e}}}
\nc{\bg}{{\mathbf{g}}}
\nc{\bh}{{\mathbf{h}}}
\nc{\bi}{{\mathbf{i}}}
\nc{\bj}{{\mathbf{j}}}
\nc{\bk}{{\mathbf{k}}}
\nc{\bl}{{\mathbf{l}}}
\nc{\bm}{{\mathbf{m}}}
\nc{\bn}{{\mathbf{n}}}
\nc{\bo}{{\mathbf{o}}}
\nc{\bp}{{\mathbf{p}}}
\nc{\bq}{{\mathbf{q}}}
\nc{\br}{{\mathbf{r}}}
\nc{\bs}{{\mathbf{s}}}
\nc{\bt}{{\mathbf{t}}}
\nc{\bu}{{\mathbf{u}}}
\nc{\bv}{{\mathbf{v}}}
\nc{\bw}{{\mathbf{w}}}
\nc{\bx}{{\mathbf{x}}}
\nc{\by}{{\mathbf{y}}}
\nc{\bz}{{\mathbf{z}}}

\nc{\fA}{{\mathfrak{A}}}
\nc{\fB}{{\mathfrak{B}}}
\nc{\fC}{{\mathfrak{C}}}
\nc{\fD}{{\mathfrak{D}}}
\nc{\fE}{{\mathfrak{E}}}
\nc{\fF}{{\mathfrak{F}}}
\nc{\fG}{{\mathfrak{G}}}
\nc{\fH}{{\mathfrak{H}}}
\nc{\fI}{{\mathfrak{I}}}
\nc{\fJ}{{\mathfrak{J}}}
\nc{\fK}{{\mathfrak{K}}}
\nc{\fL}{{\mathfrak{L}}}
\nc{\fM}{{\mathfrak{M}}}
\nc{\fN}{{\mathfrak{N}}}
\nc{\fO}{{\mathfrak{O}}}
\nc{\fP}{{\mathfrak{P}}}
\nc{\fQ}{{\mathfrak{Q}}}
\nc{\fR}{{\mathfrak{R}}}
\nc{\fS}{{\mathfrak{S}}}
\nc{\fT}{{\mathfrak{T}}}
\nc{\fU}{{\mathfrak{U}}}
\nc{\fV}{{\mathfrak{V}}}
\nc{\fW}{{\mathfrak{W}}}
\nc{\fX}{{\mathfrak{X}}}
\nc{\fY}{{\mathfrak{Y}}}
\nc{\fZ}{{\mathfrak{Z}}}

\nc{\fa}{{\mathfrak{a}}}
\nc{\fb}{{\mathfrak{b}}}
\nc{\fc}{{\mathfrak{c}}}
\nc{\fd}{{\mathfrak{d}}}
\nc{\fe}{{\mathfrak{e}}}
\nc{\ff}{{\mathfrak{f}}}
\nc{\fg}{{\mathfrak{g}}}
\nc{\fh}{{\mathfrak{h}}}
\nc{\fj}{{\mathfrak{j}}}
\nc{\fk}{{\mathfrak{k}}}
\nc{\fl}{{\mathfrak{l}}}
\nc{\fm}{{\mathfrak{m}}}
\nc{\fn}{{\mathfrak{n}}}
\nc{\fo}{{\mathfrak{o}}}
\nc{\fp}{{\mathfrak{p}}}
\nc{\fq}{{\mathfrak{q}}}
\nc{\fr}{{\mathfrak{r}}}
\nc{\fs}{{\mathfrak{s}}}
\nc{\ft}{{\mathfrak{t}}}
\nc{\fu}{{\mathfrak{u}}}
\nc{\fv}{{\mathfrak{v}}}
\nc{\fw}{{\mathfrak{w}}}
\nc{\fx}{{\mathfrak{x}}}
\nc{\fy}{{\mathfrak{y}}}
\nc{\fz}{{\mathfrak{z}}}

\nc{\sA}{{\mathsf{A}}}
\nc{\sB}{{\mathsf{B}}}
\nc{\sC}{{\mathsf{C}}}
\nc{\sD}{{\mathsf{D}}}
\nc{\sE}{{\mathsf{E}}}
\nc{\sF}{{\mathsf{F}}}
\nc{\sG}{{\mathsf{G}}}
\nc{\sH}{{\mathsf{H}}}
\nc{\sI}{{\mathsf{I}}}
\nc{\sJ}{{\mathsf{J}}}
\nc{\sK}{{\mathsf{K}}}
\nc{\sL}{{\mathsf{L}}}
\nc{\sM}{{\mathsf{M}}}
\nc{\sN}{{\mathsf{N}}}
\nc{\sO}{{\mathsf{O}}}
\nc{\sP}{{\mathsf{P}}}
\nc{\sQ}{{\mathsf{Q}}}
\nc{\sR}{{\mathsf{R}}}
\nc{\sS}{{\mathsf{S}}}
\nc{\sT}{{\mathsf{T}}}
\nc{\sU}{{\mathsf{U}}}
\nc{\sV}{{\mathsf{V}}}
\nc{\sW}{{\mathsf{W}}}
\nc{\sX}{{\mathsf{X}}}
\nc{\sY}{{\mathsf{Y}}}
\nc{\sZ}{{\mathsf{Z}}}

\nc{\sa}{{\mathsf{a}}}
\nc{\sd}{{\mathsf{d}}}
\nc{\se}{{\mathsf{e}}}
\nc{\sg}{{\mathsf{g}}}
\nc{\sh}{{\mathsf{h}}}
\nc{\si}{{\mathsf{i}}}
\nc{\sj}{{\mathsf{j}}}
\nc{\sk}{{\mathsf{k}}}
\nc{\sm}{{\mathsf{m}}}
\nc{\sn}{{\mathsf{n}}}
\nc{\so}{{\mathsf{o}}}
\nc{\sq}{{\mathsf{q}}}
\nc{\sr}{{\mathsf{r}}}
\nc{\st}{{\mathsf{t}}}
\nc{\su}{{\mathsf{u}}}
\nc{\sv}{{\mathsf{v}}}
\nc{\sw}{{\mathsf{w}}}
\nc{\sx}{{\mathsf{x}}}
\nc{\sy}{{\mathsf{y}}}
\nc{\sz}{{\mathsf{z}}}

\nc{\oA}{{\overline{A}}}
\nc{\oB}{{\overline{B}}}
\nc{\oC}{{\overline{C}}}
\nc{\oD}{{\overline{D}}}
\nc{\oE}{{\overline{E}}}
\nc{\oF}{{\overline{F}}}
\nc{\oG}{{\overline{G}}}
\nc{\oH}{{\overline{H}}}
\nc{\oI}{{\overline{I}}}
\nc{\oJ}{{\overline{J}}}
\nc{\oK}{{\overline{K}}}
\nc{\oL}{{\overline{L}}}
\nc{\oM}{{\overline{M}}}
\nc{\oN}{{\overline{N}}}
\nc{\oO}{{\overline{O}}}
\nc{\oP}{{\overline{P}}}
\nc{\oQ}{{\overline{Q}}}
\nc{\oR}{{\overline{R}}}
\nc{\oS}{{\overline{S}}}
\nc{\oT}{{\overline{T}}}
\nc{\oU}{{\overline{U}}}
\nc{\oV}{{\overline{V}}}
\nc{\oW}{{\overline{W}}}
\nc{\oX}{{\overline{X}}}
\nc{\oY}{{\overline{Y}}}
\nc{\oZ}{{\overline{Z}}}

\nc{\oa}{{\overline{a}}}
\nc{\ob}{{\overline{b}}}
\nc{\oc}{{\overline{c}}}
\nc{\od}{{\overline{d}}}
\nc{\of}{{\overline{f}}}
\nc{\og}{{\overline{g}}}
\nc{\oh}{{\overline{h}}}
\nc{\oi}{{\overline{i}}}
\nc{\oj}{{\overline{j}}}
\nc{\ok}{{\overline{k}}}
\nc{\ol}{{\overline{l}}}
\nc{\om}{{\overline{m}}}
\nc{\on}{{\overline{n}}}
\nc{\oo}{{\overline{o}}}
\nc{\op}{{\overline{p}}}
\nc{\oq}{{\overline{q}}}
\nc{\os}{{\overline{s}}}
\nc{\ot}{{\overline{t}}}
\nc{\ou}{{\overline{u}}}
\nc{\ov}{{\overline{v}}}
\nc{\ow}{{\overline{w}}}
\nc{\ox}{{\overline{x}}}
\nc{\oy}{{\overline{y}}}
\nc{\oz}{{\overline{z}}}

\nc{\tA}{{\tilde{A}}}
\nc{\tB}{{\tilde{B}}}
\nc{\tC}{{\tilde{C}}}
\nc{\tD}{{\tilde{D}}}
\nc{\tE}{{\tilde{E}}}
\nc{\tF}{{\tilde{F}}}
\nc{\tG}{{\tilde{G}}}
\nc{\tH}{{\tilde{H}}}
\nc{\tI}{{\tilde{I}}}
\nc{\tJ}{{\tilde{J}}}
\nc{\tK}{{\tilde{K}}}
\nc{\tL}{{\tilde{L}}}
\nc{\tM}{{\tilde{M}}}
\nc{\tN}{{\tilde{N}}}
\nc{\tO}{{\tilde{O}}}
\nc{\tP}{{\tilde{P}}}
\nc{\tQ}{{\tilde{Q}}}
\nc{\tR}{{\tilde{R}}}
\nc{\tS}{{\tilde{S}}}
\nc{\tT}{{\tilde{T}}}
\nc{\tU}{{\tilde{U}}}
\nc{\tV}{{\tilde{V}}}
\nc{\tW}{{\tilde{W}}}
\nc{\tX}{{\tilde{X}}}
\nc{\tY}{{\tilde{Y}}}
\nc{\tZ}{{\tilde{Z}}}

\nc{\tfD}{{\tilde{\fD}}}
\nc{\tcA}{{\tilde{\cA}}}
\nc{\tcB}{{\tilde{\cB}}}
\nc{\tcC}{{\tilde{\cC}}}
\nc{\tcD}{{\tilde{\cD}}}
\nc{\tcE}{{\tilde{\cE}}}
\nc{\tcF}{{\tilde{\cF}}}
\nc{\tcM}{{\tilde{\cM}}}
\nc{\tcP}{{\tilde{\cP}}}
\nc{\tcT}{{\tilde{\cT}}}

\nc{\ta}{{\tilde{a}}}
\nc{\tb}{{\tilde{b}}}
\nc{\tc}{{\tilde{c}}}
\nc{\td}{{\tilde{d}}}
\nc{\te}{{\tilde{e}}}
\nc{\tf}{{\tilde{f}}}
\nc{\tg}{{\tilde{g}}}
\nc{\ti}{{\tilde{\imath}}}
\nc{\tj}{{\tilde{j}}}
\nc{\tk}{{\tilde{k}}}
\nc{\tl}{{\tilde{l}}}
\nc{\tm}{{\tilde{m}}}
\nc{\tn}{{\tilde{n}}}
\nc{\tp}{{\tilde{p}}}
\nc{\tq}{{\tilde{q}}}
\nc{\tr}{{\tilde{r}}}
\nc{\ts}{{\tilde{s}}}
\nc{\tu}{{\tilde{u}}}
\nc{\tv}{{\tilde{v}}}
\nc{\tw}{{\tilde{w}}}
\nc{\tx}{{\tilde{x}}}
\nc{\ty}{{\tilde{y}}}
\nc{\tz}{{\tilde{z}}}

\nc{\hA}{{\hat{A}}}
\nc{\hB}{{\hat{B}}}
\nc{\hC}{{\hat{C}}}
\nc{\hD}{{\hat{D}}}
\nc{\hE}{{\hat{E}}}
\nc{\hF}{{\hat{F}}}
\nc{\hG}{{\hat{G}}}
\nc{\hH}{{\hat{H}}}
\nc{\hI}{{\hat{I}}}
\nc{\hJ}{{\hat{J}}}
\nc{\hK}{{\hat{K}}}
\nc{\hL}{{\hat{L}}}
\nc{\hM}{{\hat{M}}}
\nc{\hN}{{\hat{N}}}
\nc{\hO}{{\hat{O}}}
\nc{\hP}{{\hat{P}}}
\nc{\hQ}{{\hat{Q}}}
\nc{\hR}{{\hat{R}}}
\nc{\hS}{{\hat{S}}}
\nc{\hT}{{\hat{T}}}
\nc{\hU}{{\hat{U}}}
\nc{\hV}{{\hat{V}}}
\nc{\hW}{{\hat{W}}}
\nc{\hX}{{\widehat{X}}}
\nc{\hY}{{\hat{Y}}}
\nc{\hZ}{{\hat{Z}}}

\nc{\ha}{{\hat{a}}}
\nc{\hb}{{\hat{b}}}
\nc{\hc}{{\hat{c}}}
\nc{\hd}{{\hat{d}}}
\nc{\he}{{\hat{e}}}
\nc{\hg}{{\hat{g}}}
\nc{\hh}{{\hat{h}}}
\nc{\hi}{{\hat{i}}}
\nc{\hj}{{\hat{j}}}
\nc{\hk}{{\hat{k}}}
\nc{\hl}{{\hat{l}}}
\nc{\hm}{{\hat{m}}}
\nc{\hn}{{\hat{n}}}
\nc{\ho}{{\hat{o}}}
\nc{\hp}{{\hat{p}}}
\nc{\hq}{{\hat{q}}}
\nc{\hr}{{\hat{r}}}
\nc{\hs}{{\hat{s}}}
\nc{\hu}{{\hat{u}}}
\nc{\hv}{{\hat{v}}}
\nc{\hw}{{\hat{w}}}
\nc{\hx}{{\hat{x}}}
\nc{\hy}{{\hat{y}}}
\nc{\hz}{{\hat{z}}}

\nc{\hcC}{{\widehat{\cC}}}
\nc{\hcT}{{\widehat{\cT}}}

\author{Alexey Elagin}
\address{School of Mathematical and Physical Sciences, University of Sheffield, S3 7RH, Sheffield, UK}
\email{alexey.elagin@gmail.com}

\title{A categorical Torelli theorem for quartic del Pezzo surfaces}

\begin{document}

\begin{abstract}
We solve categorical Torelli problem for quartic del Pezzo surfaces. That is, we prove that a del Pezzo surface of degree $4$ can be canonically reconstructed from its Kuznetsov component, which is the orthogonal subcategory to the structure sheaf in the derived category of the surface. Our methods work in equivariant setting and over arbitrary perfect fields. Using recent theory of atomic semi-orthogonal decompositions~\cite{ElaginSchneiderShinder},  we conclude that two minimal quartic del Pezzo surfaces are birational if and only if they are isomorphic. We also verify that the Kuznetsov component of a minimal quartic del Pezzo surface is semi-orthogonally indecomposable, confirming a conjecture by Auel and Bernardara~\cite{AuelBernardara}.
\end{abstract}

\maketitle

\tableofcontents

\section{Introduction}

The bounded derived category of coherent sheaves $\Db(X)$ on an algebraic variety $X$ is a fundamental invariant of $X$, whose systematic study as such dates back to the 1980s. 
The category $\Db(X)$ encodes a substantial amount of geometric information about $X$. In particular, by~\cite{BondalOrlov2001_Reconstruction}, a smooth projective variety $X$ can be reconstructed from $\Db(X)$ provided that the canonical divisor $K_X$ is either ample or anti-ample. Nevertheless, in general, non-isomorphic varieties may have equivalent derived categories; see, for example,~\cite{Mukai1981_Duality}.

In this note we address the problem of reconstructing a variety $X$ not from the entire derived category $\Db(X)$, but from a certain admissible subcategory $\bA_X \subset \Db(X)$, commonly referred to in the literature as the \emph{residual component} or the \emph{Kuznetsov component}. Although there is no uniform definition of this component, it is known to exist for many Fano varieties and is believed to be the “essential part’’ of the derived category.

The following question is known as the \emph{categorical Torelli problem}: 
\begin{quote}
Given two varieties $X$ and $Y$ such that the categories $\bA_X$ and $\bA_Y$ are equivalent, is it true that $X$ and $Y$ are isomorphic?    
\end{quote}
Significant progress on this problem has been achieved in recent years. In particular, an affirmative answer has been obtained for many Fano hypersurfaces~\cite{LinRennemoZhang2024,LinZhang2025_Serre-algebra-matrix-factorisation,Pirozhkov2024_CategoricalTorelli}, for certain Fano threefolds~\cite{JacovskisLiuZhang2022, JacovskisLinLiuZhang2024_CategoricallTorelliGushelMukai}, and for Enriques surfaces~\cite{LiNuerStellariZhao2021_RefinedDerivedTorelli,LiStellariZhao2022_RefinedDerivedTorelli}. We refer to~\cite{PertusiStellari2023_CategoricalTorelliTheorems} for a survey of these and some other related results.

In this note we solve the categorical Torelli problem for del Pezzo surfaces of degree $4$. Note that for del Pezzo surfaces of degree $\ge 5$ there is no common understanding of what the Kuznetsov component is; moreover, these surfaces have no moduli, so the categorical Torelli problem in this case is not particularly meaningful. Del Pezzo surfaces of degree $3$ are cubic hypersurfaces in $\P^3$, and, over algebraically closed fields, the categorical Torelli problem admits an affirmative solution by the general results mentioned above, see~\cite[Cor.\ 1.6]{LinZhang2025_Serre-algebra-matrix-factorisation}. We will address the categorical Torelli problem for del Pezzo surfaces of degrees $2$ and $1$ in a subsequent work.

By definition, the Kuznetsov component of a quartic del Pezzo surface $X$ is the right orthogonal subcategory $\bA_X = \cO_X^\perp$ to the structure sheaf in the bounded derived category of coherent sheaves on $X$. 

Assume first that the base field is algebraically closed. In this case, there is a description of $\bA_X$ due to Bondal--Orlov and Kuznetsov~\cite{BondalOrlov2002_ICM, Kuznetsov2008_Derived-categories-of-quadric-fibrations}: the category $\bA_X$ is equivalent to the bounded derived category of coherent sheaves on a weighted projective line $\P_X$ canonically associated with $X$. This weighted projective line is given by the pencil of quadrics in $\P^4$ containing the anticanonical image of $X$, it carries five weighted points of weight $2$, corresponding to the degenerate quadrics in the pencil.

The facts that the category $\bA_X \cong \Db(\coh \P_X)$ determines the weighted projective line $\P_X$ up to isomorphism, and that $\P_X$ in turn determines $X$ up to isomorphism, are likely known to experts and/or can be obtained by standard technique. What appears to be new in our approach is that the reconstruction of $X$ from $\bA_X$ can be carried out in a \emph{canonical} manner: to any equivalence of components $\bA_X \xra{\sim} \bA_Y$ we can naturally associate an isomorphism of surfaces $X \xra{\sim} Y$.

To make precise what we mean by “naturally associate” such an isomorphism, we invoke the notion of a \emph{heavily separable} functor \cite{ArdizzoniMenini}. A functor $\Phi \colon \bC \to \bD$ is called heavily separable (Definition~\ref{def_hseparable}) if, for all objects $X,Y \in \bC$, there exist maps
\[
\Hom(\Phi(X),\Phi(Y)) \longrightarrow \Hom(X,Y)
\]
that are left inverses to the maps
\[
\Hom(X,Y) \longrightarrow \Hom(\Phi(X),\Phi(Y))
\]
induced by $\Phi$, and that are compatible with composition.

Our first main result is the following theorem (which remains of independent interest even in the case where the group $G$ is trivial). We emphasize that we do not impose any $\kk$-linearity assumptions on the group actions, isomorphisms, or equivalences under consideration.
\begin{theoremintro}[Theorem \ref{th_mainG}]
\label{theorem_introA}
Let $\kk$ be an algebraically closed field of $\har\ne 2$ and $G$ be a group. Then the natural functor from the  category of $G$-del Pezzo surfaces of degree $4$ over $\kk$ and $G$-isomorphisms to the category of triangulated $\kk$-linear categories with a $G$-action and $G$-equivalences, sending a surface $X$ to its Kuznetsov component $\bA_X$, is heavily separable. In particular, if for two such $G$-surfaces $X,Y$ the categories $\bA_X,\bA_Y$ are $G$-equivalent, then $X,Y$ are $G$-isomorphic.  
\end{theoremintro}

In contrast to most results on the categorical Torelli problem, we do not restrict to algebraically closed fields.
Using Theorem~\ref{theorem_introA} and  Galois descent, we deduce that the reconstruction of $X$ from $\bA_X$ is possible for del Pezzo surfaces of degree $4$ over an arbitrary perfect field.

\begin{theoremintro}[Theorem \ref{theorem_main-perfect}]
\label{theorem_introB}
Let $\kk$ be a perfect field of $\har\ne 2$. Let $X,Y$ be two del Pezzo surfaces of degree $4$ over $\kk$. If there is a $\kk$-linear equivalence $\bA_X\xra{\sim}\bA_Y$ of Fourier--Mukai type, then the surfaces $X,Y$ are isomorphic over $\kk$.  
\end{theoremintro}

In particular, our results apply to the case when $X$ is a \emph{minimal} del Pezzo surface over a perfect field $\kk$: that is, when $\Pic X$ has rank $1$. The subcategory $\bA_X$ in this case was studied in~\cite{AuelBernardara} as the \emph{Griffiths--Kuznetsov component} of $X$. In a recent work~\cite{ElaginSchneiderShinder}, the subcategory $\bA_X$ was called the \emph{atom} of $X$ and was shown to be a birational invariant of minimal del Pezzo  surfaces of degree~$4$. Moreover, it was shown that any surface birational to $X$ has a copy of $\bA_X$ as an admissible subcategory in its derived category. Combining this with Theorem~\ref{theorem_introB}, we get the following observation:
\begin{theoremintro}[Theorem \ref{th_birational-perfect}, see also~{\cite[Th. 5.5]{ElaginSchneiderShinder}}]
\label{theorem_introC}
Let $\kk$ be a perfect field of $\har \ne 2$, and let $X,Y$ be minimal del Pezzo surfaces of degree $4$ over $\kk$. 
Then the following conditions are equivalent:
\begin{enumerate}[label=(\alph*)]
    \item the surfaces $X$ and $Y$ are  isomorphic over $\kk$;
    \item the surfaces $X$ and $Y$ are  birational over $\kk$;
    \item the categories $\bA_X$ and $\bA_Y$ are $\kk$-linearly equivalent via an equivalence of Fourier--Mukai type.
\end{enumerate}
\end{theoremintro}
This result confirms the conjecture~\cite[Conj. 5.29]{ElaginSchneiderShinder}, which states that two geometrically rational surfaces are birational if and only if they have the same non-trivial atoms. 
Recently, the equivalence (a) $\Longleftrightarrow$ (b) has been obtained independently by classical methods of birational geometry in~\cite{ShramovTrepalin}.

Finally, as a result of independent interest, we prove the following.
\begin{theoremintro}[Theorems \ref{th_indecomposability}, \ref{th_indecomposability2}]
    \label{theorem_introD}
 Let $\kk$ be a perfect field of $\har \ne 2$, and let  $X$ be a minimal del Pezzo surface of degree $4$ over $\kk$. Then the category $\bA_X$ is semi-orthogonally indecomposable.
 Moreover, let $f\colon X\to C$ be a relatively minimal conic bundle over a smooth curve. If $C$ is rational, suppose that $K_X^2\le 4$. Then the category $\bA_{X/C}:=\ker f_*\subset \Db(X)$  is semi-orthogonally indecomposable.
\end{theoremintro}
This result confirms the conjecture \cite[Conj. 5.6]{AuelBernardara}, which claims  indecomposability of $\cO_X^\perp\subset \Db(X)$ for any minimal del Pezzo surface $X$ of degree $\le 4$. It also justifies the name ``atom'' (which is Greek for indecomposable)  used  for the categories $\bA_X$, $\bA_{X/C}$ in~\cite{ElaginSchneiderShinder}. 

\medskip
Let us now discuss the proofs. For a del Pezzo surface $X$ of degree $4$, we use the equivalence $\bA_X\cong \Db(\coh\P_X)$ and employ results of~\cite{LenzingMeltzer} describing the autoequivalence groups of the categories $\coh\P_X$ and $\Db(\coh\P_X)$. These results imply that any equivalence $\Db(\coh\P_X)\xra{\sim}\Db(\coh\P_Y)$, up to some shift, sends $\coh\P_X$ to $\coh\P_Y$, and any equivalence $\coh\P_X\to\coh\P_Y$ induces an isomorphism $\P_X\to\P_Y$. It is not hard to see  then (Lemma~\ref{lemma_diagonal}) that there is an isomorphism $X\xra{\sim}Y$. 

To prove Theorem~\ref{theorem_introA}, we compare the automorphism group of $X$  with the autoequivalence group of $\bA_X$. The key statement (Lemma~\ref{lemma_main}) is that the natural group homomorphism $\Aut(X)\to\Aut(\bA_X)$ is split, i.e., has a left inverse. To prove this statement, we  observe that these two groups act compatibly on the weighted projective line $\P_X$, on the set $H_X\subset \Pic X$ of ten divisor classes of conic bundles on $X$, and on the set of five degenerate quadrics containing $X$. For $\Aut (X)$ this is obvious, while for $\Aut(\bA_X)$ such actions are due to the facts that $\bA_X\cong \Db(\coh\P_X)$ and that $\P_X$, $H_X$, and $Q_X$ are intrinsic to the abelian category $\coh\P_X$.  From these actions we obtain group homomorphisms 
$$\alpha\colon \Aut(X)\to \Aut^{\ev}(H_X,Q_X)\times_{\Aut(Q_X)}\Aut(\P_X)$$ 
and 
$$\beta\colon \Aut(\bA_X)\to \Aut(H_X,Q_X)\times_{\Aut(Q_X)}\Aut(\P_X),$$
where $\Aut(H_X,Q_X)$ denotes the group of permutations of $H_X$ compatible with the natural map $H_X\to Q_X$, and 
$\Aut^{\ev}(H_X,Q_X)$ denotes the subgroup of even permutations. To construct the desired splitting, we observe that $\alpha$ is an isomorphism and the embedding of groups $\Aut^{\ev}(H_X,Q_X)\to \Aut(H_X,Q_X)$ is a split homomorphism.

Given that the homomorphism $\Aut(X)\to\Aut(\bA_X)$ splits, the proof of Theorem~\ref{theorem_introA} is purely by abstract category theory, see Lemma~\ref{lemma_hseparable}. Theorem~\ref{theorem_introB} follows from Theorem~\ref{theorem_introA} by Galois descent, and Theorem~\ref{theorem_introC} is a combination of Theorem~\ref{theorem_introB} and results from~\cite{ElaginSchneiderShinder}.

Theorem~\ref{theorem_introD} is based on results from our paper~\cite{Elagin_WPC}, where triangulated subcategories on weighted projective curves are studied. The main result that we use is that most admissible subcategories on a weighted projective curve are generated by   uniquely determined $\Hom$-orthogonal exceptional collections  of sheaves.

The structure of the paper is as follows. In Section~\ref{section_prelim} we collect the necessary facts about quartic del Pezzo surfaces, their automorphisms, categories of coherent sheaves on weighted projective lines, and their  autoequivalences. Section~\ref{section_prelim}  culminates in proving that the natural homomorphism $\Aut(X)\to\Aut(\bA_X)$ splits. We also discuss heavily separable functors in this section. After proper preparation, we prove Theorems~\ref{theorem_introA},~\ref{theorem_introB},~\ref{theorem_introC} in Section~\ref{section_main}.  Section~\ref{section_indecomposability} is relatively self-contained and devoted to the proof of Theorem~\ref{theorem_introD}.

\subsection{Acknowledgements}

This paper grew out of  collaboration with Evgeny Shinder and Julia Schneider on atomic semi-orthogonal decompositions~\cite{ElaginSchneiderShinder}, and I am greatly indebted to them. I also thank Vanya Cheltsov, Alexander Kuznetsov, and Constantin Shramov for useful discussions and their interest in this work. The study was supported by the UKRI Horizon Europe guarantee award `Motivic invariants and birational geometry of simple normal crossing degenerations' EP/Z000955/1.

\subsection{Conventions}

We work over a perfect field $\kk$ of characteristic $\ne 2$ which is not necessarily algebraically closed. Algebraic varieties are supposed to be smooth projective and over $\kk$, but morphisms are not supposed to be morphisms over $\kk$ by default. Similarly, functors between $\kk$-linear categories are always additive but not necessarily $\kk$-linear. Functors between triangulated categories are always supposed to be exact. When we consider group actions/$G$-varieties, we make no implicit assumptions, in particular, groups can be infinite  or act non-faithfully.

\section{Preliminaries}
\label{section_prelim}

\subsection{Coxeter groups of type $B_5$ and $D_5$}

Let $H$ be a finite set of $10$ elements, $Q$ be a finite set of $5$ elements, and $\pi\colon H\to Q$ be a $2$-to-$1$ map.
Denote by $\Aut(H)\cong S_{10}$ and $\Aut(Q)\cong S_5$ the groups of permutations of $H$ and $Q$, respectively, denote by $\Aut(H,Q)$ the subgroup $\{(f,g)\mid \pi f=g\pi\}\subset \Aut(H)\times \Aut(Q)$ of compatible permutations of $H$ and $Q$, and by  $\Aut(H/Q)\cong (\Z/2\Z)^5$ the subgroup $\{f\mid \pi f=\pi\}\subset \Aut(H)$ of permutations of $H$ over $Q$. There is an exact sequence of groups
\begin{equation}
\label{eq_sesB5}
    \{e\}\to \Aut(H/Q)\to \Aut(H,Q)\to \Aut(Q)\to \{e\},
\end{equation}
given by the maps $(f)\mapsto (f,e)$ and $(f,g)\mapsto g$. 
Let $\Aut^{\ev}(H,Q)$ be the subgroup $\{(f,g)\mid \text{$f$ is even})\}\subset \Aut(H,Q)$ and $\Aut^{\ev}(H/Q)\cong (\Z/2\Z)^4$ be the subgroup of even permutations in $\Aut(H/Q)$.
These groups form an exact sequence, which is a subsequence of~\eqref{eq_sesB5}
\begin{equation}
    \{e\}\to \Aut^{\ev}(H/Q)\to \Aut^{\ev}(H,Q)\to \Aut(Q)\to \{e\}.
\end{equation}
Note that the groups $\Aut(H,Q)$ and $\Aut^{\ev}(H,Q)$ are isomorphic to Coxeter groups of types $B_5$ and $D_5$, respectively.

The following easy lemma will be important for us.
\begin{lemma}
    \label{lemma_BDsplit}
The inclusion $\Aut^{\ev}(H,Q)\to \Aut(H,Q)$ splits as a group homomorphism over $\Aut(Q)$. 
\end{lemma}
\begin{proof}
    Let $c\in \Aut (H/Q)$ be the product of five transpositions in the fibres of $\pi$, it is an odd central element of order $2$. Then $\Aut(H,Q)$ is the direct product of its subgroups $\Aut^{\ev}(H,Q)\times \langle c\rangle$, and the projection $\Aut(H,Q)\to \Aut^{\ev}(H,Q)$ gives the desired splitting. 
\end{proof}

\subsection{Weighted projective lines  and their autoequivalences}
We refer to~\cite{GL},~\cite{Lenzing_hercat}, and~\cite{LenzingMeltzer} for the definitions and necessary results about weighted projective lines, their categories of coherent sheaves, and the autoequivalences of these categories.
In this subsection we assume that the base field $\kk$ is algebraically closed.

A weighted projective line $\P=(P,w)$ over $\kk$ consists of a projective line $P\cong \P^1_\kk$ over $\kk$ together with a weight function 
$w\colon P(\kk)\to \Z_{\ge 1}$ such that $w(p)=1$ for all but finitely many points $p\in P(\kk)$. Points $p_1,\ldots, p_n$ where $w>1$ are called \emph{weighted points} and the integers $w_i=w(p_i)$, $i=1,\ldots,n$ are called the \emph{weights} at these points. The collection of weights $(w_1,\ldots,w_n)$ is called the \emph{type} of $\P$. An \emph{isomorphism of weighted projective lines} $(P,w)\xra{\sim}(P',w')$ is an isomorphism $f\colon P\xra{\sim}P'$ of varieties that preserves the weights, that is, such that $w=w'f$. We denote by $\Aut(\P)$ the  automorphism group of $\P$ and by $\Aut_\kk(\P)$ the subgroup consisting of $\kk$-linear automorphisms.
\begin{remark}
\label{remark_AutPfinite}
    Note that $\Aut_\kk(\P)$ is a subgroup of $\Aut_\kk(\P^1_\kk)\cong \PGL_2(\kk)$. Moreover, if the number~$n$ of weighted points of $\P$ is at least three, then $\Aut_\kk(\P)$ is a subgroup of $\Aut(\{p_1,\ldots,p_n\})$, hence is finite. Indeed, a $\kk$-linear automorphism of $\P^1_\kk$ is determined by the image of any three points.
\end{remark}

The category of coherent sheaves on $\P=(P,w)$ can be described following Geigle and Lenzing, see~\cite{GL}. 
Let $V=\rH^0(P,\cO(1))$, choose non-zero linear forms $u_i\in V$ such that $u_i(p_i)=0$. Denote
$$S_\P:=(S^\bul(V)\otimes \kk[U_1,\ldots,U_n])/(U_1^{w_1}-u_1,\ldots,U_n^{w_n}-u_n),$$
where $S^\bul(V)$ denotes the symmetric algebra.
Let $\L=\L_\P$ be the abelian group generated by elements $\bar c,\bar x_1,\ldots,\bar x_n$ and relations $\bar c=w_i\cdot \bar x_i$ for $i=1,\ldots,n$. Then $S_\P$ is an $\L$-graded commutative algebra with the grading given by $\deg(U_i)=\bar x_i$, $\deg (v)=\bar c$ for $v\in V$.
Define $\coh\P$ as the Serre quotient
$$\coh\P=\frac{\mathrm{mod}^\L{-}S_\P}{\mathrm{mod}_0^\L{-}S_\P},$$
where $\mathrm{mod}^\L{-}S_\P$ is the category of finitely generated $\L$-graded $S_\P$-modules, and $\mathrm{mod}_0^\L{-}S_\P\subset \mathrm{mod}^\L{-}S_\P$ is the subcategory of finite-dimensional modules. Objects of $\coh\P$ are called \emph{coherent sheaves} on $\P$. For a sheaf $F$ and an element $\bar x\in \L$, denote by $F(\ox)$ the sheaf obtained from $F$ by twisting with $\ox$: that is, by shifting the grading by $\ox$. 
The category $\coh\P$ is a $\kk$-linear abelian hereditary $\Ext$-finite category. It is  
equipped with a translation functor $\tau\colon \coh\P\to\coh\P$ given by twisting with the canonical element 
$$\bar\omega:=-2\oc+\sum_{i=1}^n(w_i-1)\ox_i.$$
The shift $\tau[1]$ of the translation functor serves as a Serre functor on the derived category $\Db(\coh\P)$, so $\tau$ is intrinsic to the category $\coh\P$.

Let $\coh_0\P\subset \coh\P$ be the full subcategory of torsion sheaves, that is, of objects of finite length. There is a decomposition
$$\coh_0\P=\coprod_{p\in P(\kk)}\coh_p\P,$$
where $\coh_p\P$ is the full subcategory of sheaves supported at the closed point $p\in P$. Each category $\coh_p\P$ has $w(p)$ simple objects and they form one orbit under $\tau$. Therefore, the set of closed points of $P$ is identified with the set of $\tau$-orbits of simple objects in $\coh\P$, and the multiplicities are recovered as the lengths of the orbits. Moreover, if a simple sheaf $F$ is supported at a weighted point $p_i$ then all simple sheaves supported at $p_i$ are $F(k\ox_i)$, $k=0,\ldots,w_i-1$, and $F\cong F(\oc)\cong F(\ox_j)$ for $i\ne j$. We remark that $\P$ is canonically determined by $\coh\P$ in the following sense.
\begin{lemma}
\label{lemma_new}
For any weighted projective lines $\P,\P'$ over $\kk$,
any equivalence $\phi\colon \coh\P\xra{\sim}\coh\P'$ defines canonically an isomorphism $f=\phi_\P\colon \P\xra{\sim}\P'$, which is  compatible with $\phi$  via  the action on  simple sheaves: for any point $p\in P(\kk)$ one has $\phi(\coh_p\P)=\coh_{f(p)}\P'$ as subcategories in $\coh\P'$. Moreover, if $\phi$ is linear over $\kk$ then $f$ also is.
\end{lemma}
\begin{proof}
    See the proof of~\cite[Prop. 3.1, Th. 3.4]{LenzingMeltzer}. They assume that $\P=\P'$ and $\phi$ is $\kk$-linear but the proof works in a more general setting as well.
\end{proof}

We will be interested in autoequivalences of the abelian category $\coh\P$. Denote by $\Aut(\coh\P)$ the group of isomorphism classes of autoequivalences of $\coh\P$, and by $\Aut_\kk(\coh\P)$ the subgroup formed by $\kk$-linear autoequivalences. First, there are geometric equivalences: any isomorphism $f\colon \P\xra{\sim}\P'$ of weighted projective lines induces an equivalence $\coh\P\xra{\sim}\coh \P'$ given by the direct image functor. There are also geometrically trivial autoequivalences: for any $\ox\in\L$, the shift of grading by $\ox$ defines an autoequivalence of $\coh\P$, which does not change the support of sheaves. These two types of autoequivalences generate the whole group $\Aut(\coh\P)$: 
\begin{lemma}
    \label{lemma_AutCohP}
    There are natural split exact sequences of groups
    \begin{gather*}
        0\to \L\to \Aut(\coh \P)\to \Aut(\P)\to\{e\},\\
        0\to \L\to \Aut_\kk(\coh \P)\to \Aut_\kk(\P)\to\{e\}.
    \end{gather*}
\end{lemma}
\begin{proof}
This is~\cite[Th. 3.4]{LenzingMeltzer} for $\kk$-linear autoequivalences/automorphisms, but their proof works in the general case as well. 
\end{proof}

The degree homomorphism $\deg\colon \L\to\Q$ is given by $\deg(\oc)=1, \deg(\ox_i)=\frac1{w_i}$. The kernel of $\deg$ is the torsion subgroup $\L^{\tors}\subset \L$.

Properties of the category $\coh\P$ depend on  whether the degree $\deg(\bar\omega)$ of the canonical element is negative, zero, or positive. A weighted projective line is called \emph{domestic/tubular/wild} in these three cases respectively.

\medskip
Since the category $\coh\P$ is hereditary, its derived category $\Db(\coh\P)$ has a very simple structure: any object is isomorphic to the direct sum of its cohomology sheaves. In particular, any indecomposable object is a shift of an indecomposable sheaf. The group $\Aut(\Db(\coh \P))$ of isomorphism classes of autoequivalences of $\Db(\coh\P)$ is described in~\cite{LenzingMeltzer}; we will only need this description in the wild case.
\begin{lemma}
     \label{lemma_AutDbCohP}
     For a weighted projective line $\P$ not of tubular type  there is a direct product decomposition 
     $$\Aut(\Db(\coh \P))\cong \Aut(\coh\P)\times \Z,$$
     where $\Z$ acts by cohomological shifts. Moreover, if $\P,\P'$ are weighted projective lines not of tubular type and $\phi\colon \Db(\coh\P)\xra{\sim}\Db(\coh\P')$ is an equivalence, then there is some $n\in\Z$ such that $\phi[n]$ restricts to  an equivalence $\coh\P\xra{\sim}\coh\P'$. Similar statements hold for $\kk$-linear autoequivalences.
\end{lemma}
\begin{proof}
See~\cite[Prop. 4.1, Prop. 4.2]{LenzingMeltzer} for the first statement. The second  is proved in the same way as~\cite[Prop. 4.1]{LenzingMeltzer}.
\end{proof}

\medskip
From now on we restrict ourselves to weighted projective lines of type $(2,2,2,2,2)$. The degree of $\bar\omega$ is $1/2>0$ in this case, so these   weighted projective lines are wild. Let $Q_\P:=\{p_1,\ldots,p_5\}$ be the weighted points of $\P$. For any $p_i\in Q_\P$ there are two simple sheaves $F_i,F'_i$ supported at~$p_i$; denote $H_\P:=\{F_i,F'_i\}_{i=1\ldots 5}$. One has $\tau F_i\cong F'_i, \tau F'_i\cong F_i$; moreover, sheaves in $H_\P$ are the only simple objects in $\coh \P$ forming $\tau$-orbits of length $2$. Hence any autoequivalence $\phi$ of $\coh\P$ preserves $H_\P$ and is compatible with the map $\pi\colon H_\P\to Q_\P$: if $f\in\Aut(\P)$ is induced by $\phi$ then the sheaves $\phi(F_i)$, $\phi(F'_i)$ are supported at the point $f(p_i)$. We obtain the following.
\begin{lemma}
\label{lemma_AutCohP2}
For a weighted projective line $\P$ of type $(2,2,2,2,2)$ there is a commutative diagram, where the rows are exact sequences:
\begin{equation}
\label{eq_AutP}
    \xymatrix{
\{e\}\ar[r] & \L\ar[r]\ar[d] &\Aut(\coh\P)\ar[r]\ar[d] &\Aut(\P)\ar[r]\ar[d] &\{e\}\\
\{e\}\ar[r] & \Aut(H_\P/Q_\P)\ar[r] &\Aut(H_\P,Q_\P)\ar[r] &\Aut(Q_\P)\ar[r] &\{e\}.
    }
\end{equation}
\end{lemma}

\begin{remark}
Note that $\Aut_{\kk}(\P)$ is  a subgroup in $\Aut(Q_\P)\cong S_5$ by Remark~\ref{remark_AutPfinite}. Moreover, for general $\P$ the group $\Aut_{\kk}(\P)$ is trivial so that $\Aut_{\kk}(\coh \P)\cong \L$ by Lemma~\ref{lemma_AutCohP}.

Elements of finite order in $\Aut_\kk(\coh\P)$ form a subgroup denoted $ \Aut_\kk^{\tors}(\coh\P)$, which fits into an exact sequence (where $\L^{\tors}\subset \L$ is the torsion subgroup)
$$\{e\}\to \L^{\tors} \to \Aut_\kk^{\tors}(\coh\P) \to \Aut_\kk(\P)\to \{e\}, $$
which is a subsequence in the second row of~\eqref{eq_AutP}.
\end{remark}

\subsection{Quartic del Pezzo surfaces and their automorphisms}
We refer to~\cite{Manin_CubicForms} and~\cite{Dolgachev_ClassicalAlgebraicGeometry} for the basic definitions and results about del Pezzo surfaces. 
In this subsection we assume that the base field $\kk$ is algebraically closed and of characteristic $\ne 2$. 

Let $X$ be a del Pezzo surface over $\kk$ of degree $4$, that is, a smooth irreducible projective surface over $\kk$ with $-K_X$ ample and $(-K_X)^2=4$. Such a surface is isomorphic to the blow-up of $\P^2$ at five general points. Any such isomorphism defines a standard basis $H,E_1,\ldots,E_5$ in the Picard group $\Pic(X)$, where $H$ is the pull-back of a line and $E_1,\ldots,E_5$ are exceptional divisors. The intersection form in this basis is given by $H^2=1, E_i^2=-1, H\cdot E_i=E_i\cdot E_j=0$ for $i\ne j$,   and the canonical class is $K_X=-3H+E_1+\ldots+E_5$. Let $H_X\subset \Pic(X)$ be the set of $0$-classes:
$$H_X=\{h\in \Pic(X)\mid h^2=0, h\cdot K_X=-2\}.$$
Any $h\in H_X$ is represented by a smooth rational curve and the linear system $|h|$ defines a conic bundle $X\to \P^1$. There are $10$ classes in $H_X$; in the standard basis they are 
$$H-E_1,\ldots,H-E_5, 2H-E_2-E_3-E_4-E_5,\ldots,2H-E_1-E_2-E_3-E_4.$$
Divisors in $H_X$ are naturally divided into pairs: in any such pair $h,h'$ one has $h+h'=-K_X$.

The anticanonical system $|-K_X|$ defines an embedding $\phi\colon X\to \P^4=\P(\rH^0(X,\omega_X^{-1}))$, whose image is the intersection of a pencil of quadrics. More precisely, the family of quadrics containing the image of $\phi$ is one-dimensional; we will denote this family $\P(\rH^0(\P^4,I_{\phi(X)}(2)))\cong \P^1$ by $P_X$. A general quadric in $P_X$ is non-degenerate; there are five  degenerate quadrics in $P_X$, all of them of  corank $1$. We will denote the subset of degenerate quadrics in $P_X$ by $Q_X$.  We define the weighted projective line $\P_X$ as $P_X$ with weights $2$ at the five points of $Q_X$. The association $X\mapsto \P_X$ is canonical in the sense that it extends to isomorphisms between surfaces: any isomorphism $f\colon X\xra{\sim}Y$ of surfaces induces an isomorphism of projective lines $P_X\xra{\sim}P_Y$ that respects the weighted points, that is, induces an isomorphism $\P_X\to\P_Y$. The isomorphism class of $X$ is uniquely determined by the weighted projective line $\P_X$. More precisely, one has the following:
\begin{lemma}
    \label{lemma_diagonal}
Let $X,Y$ be del Pezzo surfaces of degree $4$. For any (not necessarily $\kk$-linear) isomorphism $g\colon \P_X\xra{\sim}\P_Y$ there is an isomorphism $f\colon X\xra{\sim}Y$ inducing $g$. If $g$ is $\kk$-linear then $f$ is $\kk$-linear.
\end{lemma}
\begin{proof}
First assume that $g$ is $\kk$-linear. 

Let $Q_X=\{q_{X,i}\}_{i=1\ldots 5}$,  $Q_Y=\{q_{Y,i}\}_{i=1\ldots 5}$, so that $g(q_{X,i})=q_{Y,i}$ for all $i$.
Embed $X$ anticanonically  into $\P^4$ as the intersection of a pencil of quadrics. By~\cite[Th. 8.6.2]{Dolgachev_ClassicalAlgebraicGeometry}, quadrics in the pencil are simultaneously  diagonalisable in certain coordinates. Up to renumbering of  coordinates we can assume that for $i=1,\ldots, 5$ the degenerate quadric $q_{X,i}$ is given by the equation $\sum_{j=1}^5 \lambda_{i,j}x_j^2=0$
where $\lambda_{i,i}=0$ and $\lambda_{i,j}\ne 0$ for $i\ne j$. Rescaling coordinates we can then assume that $q_{X,1}$ is given by $x_2^2+x_3^2+x_4^2+x_5^2=0$.
Finally, rescaling the coordinate $x_1$ we can additionally assume that $q_{X,2}$ is given by $x_1^2+x_3^2+\lambda_X x_4^2+\mu_X x_5^2=0$. Then the pencil of quadrics containing $X$ is given by 
\begin{equation}
\label{eq_diagonal_X}
t_0(x_1^2+x_3^2+\lambda_X x_4^2+\mu_X x_5^2)-t_1(x_2^2+x_3^2+x_4^2+x_5^2)=0, \quad (t_0:t_1)\in \P^1=P_X.
\end{equation}
The degenerate points $q_{X,1}$, $q_{X,2}$, $q_{X,3}$, $q_{X,4}$, $q_{X,5}$ of this pencil are respectively $t_1/t_0=\infty$, $0,1,\lambda_X,\mu_X$. Similarly we find a coordinate system for the anticanonical embedding of $Y$. A $\kk$-linear isomorphism $\P_X\xra{\sim}\P_Y$ is uniquely determined by the images of three points, so in the chosen coordinates on $P_X$ and $P_Y$ the morphism~$g$ must be the identity,  therefore $\lambda_X=\lambda_Y$, $\mu_X=\mu_Y$. Equations~\eqref{eq_diagonal_X} for $X$ and $Y$ are the same, hence the identity map $\P^4\to\P^4$ restricts to an isomorphism between $X$ and $Y$ compatible with $g\colon \P_X\to\P_Y$. 

Now consider an arbitrary isomorphism $g\colon \P_X\to\P_Y$. Note that the $\kk$-linear structure is not involved in the statement of the lemma, so it suffices to prove the statement for \emph{some} $\kk$-linear structures on $P_X$ and $P_Y$. The isomorphism $g$ induces an isomorphism on scalars $g_{\kk}\colon \kk\to\kk$ given as $\rH^0(P_X,\cO_{P_X})\xra{\sim}\rH^0(P_Y,\cO_{P_Y})$, so that there is a commutative diagram 
$$\xymatrix{ P_X\ar[r]^g\ar[d]  & P_Y\ar[d] \\ \Spec\kk\ar[r]^{g_\kk} & \Spec\kk.}$$
Twisting the $\kk$-linear structure on $P_Y$ we may assume that $g$ is $\kk$-linear and so $g$ is induced by an isomorphism $f\colon X\xra{\sim}Y$ (which is $\kk$-linear with respect to the twisted $\kk$-linear structure on $Y$).    
\end{proof}

There is a natural bijection between pairs of divisors in $H_X$ and points of $Q_X$. Indeed, any degenerate quadric of corank one from $Q_X$ is the cone over a non-degenerate quadric in $\P^3$ so that the vertex of the cone is not contained in $X$. Projection from the vertex gives a regular map $X\to \P^1\times \P^1$ of degree $2$, hence two conic fibrations on $X$, hence two zero-classes adding up to a hyperplane section, that is, to $-K_X$. Therefore, there is a natural $2$-to-$1$ map $$\pi_X\colon H_X\to Q_X.$$ 

Our next goal is to understand the group $\Aut(X)$ of all (not necessarily $\kk$-linear) automorphisms of $X$.
Any automorphism of $X$ induces an automorphism of $\P_X$; thus we get a   homomorphism $\Aut(X)\to \Aut(\P_X)$. Also, any automorphism of $X$ induces an automorphism of the Picard group $\Pic(X)$ preserving $K_X$ and the intersection form, hence induces an automorphism of $H_X$ compatible with $\pi_X$. One gets a homomorphism $\Aut(X)\to \Aut(H_X,Q_X)$. Its image is in $\Aut^{\ev}(H_X,Q_X)$; in fact, the image of the homomorphism 
$\Aut(\Pic(X)) \to \Aut(H_X,Q_X)$ is in $\Aut^{\ev}(H_X,Q_X)$. To see this, recall that the automorphism group of $\Pic(X)$ is the Weyl group of the root system of type $D_5$, formed by $(-2)$-classes in $K_X^\perp\subset \Pic(X)$ (see~\cite[Th. 23.9]{Manin_CubicForms}).
Let $\oH_X\subset \Pic(X)_\Q$ be the orthogonal projection of $H_X$ onto $K_X^\perp$: $\oH_X=\{h+1/2K_X\mid h\in H_X\}$. Choosing one vector in each pair in~$\oH_X$ one can write $\oH_X=\{\pm \oh_1,\ldots,\pm \oh_5\}$, where $\oh_1,\ldots,\oh_5$ is an orthonormal basis  in $\Pic(X)_\Q$ with respect to minus the intersection form. Then the roots are the vectors $\pm \oh_i\pm \oh_j$ for $i\ne j$, and reflections in these roots generate the Weyl (or Coxeter) group of type $D_5$, which acts by even permutations of the set $\oH_X$ (and linear transformations of $\Pic(X)_\Q$ inducing odd permutations of $\oH_X$ do not preserve $\Pic(X)\subset \Pic(X)_\Q$). 

Denote by $\Aut^0(X)$ the group of $\kk$-linear automorphisms of $X$ that act trivially on $P_X$.

Homomorphisms $\Aut(X)\to \Aut(\P_X)$ and 
$\Aut(X) \to \Aut^{\ev}(H_X,Q_X)$ induce  homomorphisms 
$$\alpha_0\colon \Aut^0(X)\to\Aut^{\ev}(H_X/Q_X)$$ and
$$\alpha\colon \Aut(X) \to \Aut^{\ev}(H_X,Q_X)\times_{\Aut(Q_X)} \Aut(\P_X).$$
\begin{lemma} 
\label{lemma_AutX}
In the above notation
\begin{enumerate}
    \item Homomorphism $\alpha_0$ is an isomorphism.
    \item Homomorphism $\alpha$ is an isomorphism.
    \item There is a commutative diagram where rows are exact sequences
\begin{equation}
\label{eq_AutX}
    \xymatrix{
\{e\}\ar[r] & \Aut^0(X)\ar[r]\ar[d]^{\alpha_0} &\Aut(X)\ar[r]\ar[d] &\Aut(\P_X)\ar[r]\ar[d] &\{e\}\\
\{e\}\ar[r] & \Aut^{\ev}(H_X/Q_X)\ar[r] &\Aut^{\ev}(H_X,Q_X)\ar[r] &\Aut(Q_X)\ar[r] &\{e\}.
    }
\end{equation}
\end{enumerate}
\end{lemma}
\begin{proof}
    For (1),
    take  $X$ as given by the equations~\eqref{eq_diagonal_X} and consider automorphisms of $X$ given by $x_i\mapsto  \pm x_i$ (with arbitrary sign choices). This defines a subgroup of $\Aut^0(X)$ isomorphic to $(\Z/2\Z)^4$, since such automorphisms preserve all quadrics~\eqref{eq_diagonal_X}. Recall that $\Aut^{\ev}(H_X/Q_X)$ is also isomorphic to $(\Z/2\Z)^4$; then (1) follows from the injectivity of  $\alpha$ that we check next.

    For (2), we first check that $\alpha$ is injective. Assume $f\colon X\xra{\sim}X$ is an automorphism inducing the identity on $\P_X$ and on $H_X$. Then $f$ induces the identity on $\Spec\kk$, that is, $f$ is $\kk$-linear. Also, $f$ must act trivially on $\Pic(X)$, which is generated by $H_X$ over $\Q$. Hence $f$ preserves all $(-1)$-curves on~$X$. Let $p\colon X\to \P^2$ be a contraction of  five $(-1)$-curves, then $f$ induces a $\kk$-linear automorphism on~$\P^2$ which preserves five  points in general position and hence is the identity. Therefore $f=\id_X$.

    Now we check that $\alpha$ is surjective. Let $g\in \Aut (\P_X)$ and $s\in \Aut^{\ev}(H_X,Q_X)$ be automorphisms compatible on $Q_X$. By Lemma~\ref{lemma_diagonal} there is $\tf\in\Aut(X)$ inducing $g$ on $\P_X$. We want to adjust $\tf$ to get $f\in\Aut(X)$ that induces $g$ on $\P_X$ and $s$ on $H_X$. Let $\tf_H$ be the automorphism induced by $\tf$ on $H_X$; consider the difference $(\tf_H)^{-1}s$. This difference acts trivially on $Q_X$, therefore belongs to $\Aut^{\ev}(H_X/Q_X)$. By (1), there is $u\in\Aut(X)$ inducing $(\tf_H)^{-1}s$ on $H_X$ and  the identity on $\P_X$. Then~$f=\tf u$ is the desired automorphism of $X$.

    Statement (3) follows from the above discussion.
\end{proof}

\begin{remark}
\label{remark_AutXklinear}
    Restricting to $\kk$-linear automorphisms one gets a short exact sequence 
    \begin{equation}
        \label{eq_AutkX}
        \{ e\}\to \Aut^0(X)\to \Aut_{\kk}(X)\to \Aut_{\kk}(\P_X)\to \{ e\}
    \end{equation}    
    and an isomorphism $\Aut_{\kk}(X) \xra{\sim}\Aut^{\ev}(H_X,Q_X)\times_{\Aut(Q_X)} \Aut_{\kk}(\P_X)$.
    Note that $\Aut_{\kk}(\P_X)$ is  a subgroup of $\Aut(Q_X)\cong S_5$ by Remark~\ref{remark_AutPfinite}, so that~\eqref{eq_AutkX} is a subsequence of the second line in~\eqref{eq_AutX}. Moreover, for general $X$ the group $\Aut_{\kk}(\P_X)$ is trivial so that $\Aut_{\kk}(X)=\Aut^0(X)\cong (\Z/2\Z)^4$.
\end{remark}

\subsection{Orthogonal to the structure sheaf on quartic  del Pezzo surfaces}
As before, we assume that the base field $\kk$ is algebraically closed and of characteristic $\ne 2$. 
Let $X$ be a del Pezzo surface of degree $4$ over $\kk$. 
\begin{definition}
    Define the subcategory $\bA_X\subset \Db(X)$ as the right orthogonal subcategory to the structure sheaf:
    $$\bA_X:=\cO_X^\perp=\{F\in\Db(X)\mid \Hom(\cO_X,F[i])=0\quad\text{for all $i\in\Z$}\}.$$
\end{definition}
This is an admissible  triangulated subcategory of $\Db(X)$ fitting into a semi-orthogonal decomposition
$$\Db(X)=\langle\bA_X, \langle\cO_X\rangle\rangle.$$

The following fact is well-known (see~\cite[Th. 6.1]{BondalOrlov2002_ICM},~\cite[Cor. 5.7]{Kuznetsov2008_Derived-categories-of-quadric-fibrations},  where a much more general result concerning any complete intersection of quadrics is obtained): 
there is a $\kk$-linear exact  equivalence    
$\bA_X\cong \Db(\coh\P_X)$.
We will need a more refined statement, which is proved in~\cite{Kuznetsov2008_Derived-categories-of-quadric-fibrations}: 
\begin{prop}
\label{prop_BO}
Let  $\cB_X$ be  the sheaf of even parts of Clifford algebras on the projective line $P_X$ of quadrics containing $X$. The category $\modd \cB_X$ of coherent sheaves of $\cB_X$-modules is equivalent to $\coh\P_X$.
There is a $\kk$-linear exact  equivalence    
$$\psi_X\colon\Db(\modd \cB_X)\xra{\sim}\bA_X.$$
\end{prop}
We explain below that
there is a canonically defined abelian subcategory $\cP_X$ in $\bA_X$ that is equivalent to $\modd \cB_X$ and to  $\coh\P_X$. 
\begin{lemma}
There exists a unique subcategory $\cP_X\subset \bA_X$ with the following  properties. 
\begin{enumerate}
    \item There is a $\kk$-linear exact equivalence $\Db(\modd\cB_X)\xra{\sim}\bA_X$ identifying $\modd \cB_X$ with $\cP_X$. 
    \item For any class $h\in H_X$, the sheaf $\cO_X(-h)$ belongs to $\cP_X$.
\end{enumerate}
\end{lemma}
\begin{proof}
Subcategories in $\bA_X$ satisfying (1) exist by  Proposition~\ref{prop_BO}. Moreover, they all differ by a shift. Indeed, any two equivalences 
$\Db(\modd \cB_X)\to\bA_X$ differ by an  autoequivalence of $\Db(\modd \cB_X)$, and by Lemma~\ref{lemma_AutDbCohP} any autoequivalence of $\Db(\modd \cB_X)$ sends the heart $\modd \cB_X\subset \Db(\modd \cB_X)$ of the standard t-structure to its shift. Now we choose the shift to satisfy (2). An easy computation shows that for any class $h\in H_X$ the sheaf $\cO(-h)$ belongs to $\bA_X$. Recall that any indecomposable object in $\Db(\modd \cB_X)$ is a shift of a $\cB_X$-module. Therefore, for any $h\in H_X$ one can find a  subcategory satisfying (1) and containing  $\cO(-h)$. Clearly these categories are all the same for different $h$.    
\end{proof}

Any isomorphism $f\colon X\xra{\sim}X'$ between del Pezzo surfaces of degree $4$ induces an equivalence $\Db(X)\xra{\sim}\Db(X')$, given by the direct image functor $f_*$, which restricts to an equivalence $f_{\bA}\colon \bA_X\to\bA_{X'}$.
Clearly $f_{\bA}$ sends the set of sheaves  $\cO(-h), h\in H_X$ to the set of sheaves  $\cO(-h), h\in H_{X'}$, and therefore restricts to an equivalence $f_{\cP}\colon \cP_X\to\cP_{X'}$. We have seen that there is a weighted projective line intrinsic to the abelian category $\cP_X$, and we will hence denote  this weighted projective line by $\P_{\cP_X}$: 
the closed points of $\P_{\cP_X}$ are in bijection with the $\tau$-orbits of simple objects in $\cP_X$. In particular, any equivalence $\phi\colon \cP_X\xra{\sim}\cP_{X'}$ defines a canonical isomorphism $\phi_\P\colon \P_{\cP_X}\to\P_{\cP_{X'}}$. Since the category $\cP_X$ is equivalent to $\coh\P_X$ (non-canonically), the projective lines $\P_{\cP_X}$ and $\P_X$ are isomorphic. We claim that in fact they are isomorphic canonically:
\begin{lemma}
\label{lemma_epsilon}
There are $\kk$-linear isomorphisms 
$$\epsilon_X\colon \P_X\xra{\sim}\P_{\cP_X}, $$ 
for all  del Pezzo surfaces $X$ of degree $4$, providing a commutative triangle of functors
\begin{equation*}
X\mapsto \P_X, X\mapsto \cP_X, \cP_X\mapsto \P_{\cP_X}
\end{equation*}
between the categories of del Pezzo surfaces of degree $4$, weighted projective lines  of type $(2,2,2,2,2)$, and abelian categories equivalent to coherent sheaves on such weighted projective lines. 
Explicitly, for any isomorphism  $f\colon X\xra{\sim}X'$ of del Pezzo surfaces of degree $4$, 
the following diagram is commutative:
\begin{equation}
\label{eq_PPPP}
    \xymatrix{\P_X\ar[r]^{\epsilon_X} \ar[d]^{f_\P} & \P_{\cP_X} \ar[d]^{(f_\cP)_\P} \\
\P_{X'}\ar[r]^{\epsilon_{X'}}  & \P_{\cP_{X'}}.}
\end{equation}
\end{lemma}
\begin{proof}
This follows from the proof of Proposition~\ref{prop_BO}. By~\cite{Kuznetsov2008_Derived-categories-of-quadric-fibrations}, equivalences $\psi_X$ from Proposition~\ref{prop_BO} are $\kk$-linear and natural in the following sense: for an isomorphism $f\colon X\xra{\sim}X'$, there is an isomorphism 
$(f_\P)_*\cB_X\xra{\sim}\cB_{X'}$ of sheaves of algebras on $P_{X'}$ and the direct image functor $(f_\P)_*$ induces equivalences 
$\modd \cB_X\xra{\sim}\modd ((f_\P)_*\cB_X)\xra{\sim}\modd \cB_{X'}$ which fit into a commutative diagram of equivalences:
$$\xymatrix{\Db(\modd\cB_X)  \ar[r]^-{\psi_X} \ar[d]^{(f_\P)_*}  & \bA_X \ar[d]^{f_{\bA}}\\ 
\Db(\modd\cB_{X'}) \ar[r]^-{\psi_{X'}} & \bA_{X'}.}$$
Possibly after some shift, these functors restrict to a commutative diagram of equivalences between abelian categories:
$$\xymatrix{\modd\cB_X  \ar[r] \ar[d]^{(f_\P)_*}  & \cP_X \ar[d]^{f_{\cP}}\\ 
\modd\cB_{X'} \ar[r] & \cP_{X'}.}$$
This diagram induces diagram~\eqref{eq_PPPP} of isomorphisms between associated weighted projective lines. Indeed, horizontal arrows provide identifications $\epsilon_X, \epsilon_{X'}$ of $\P_X$, $\P_{X'}$ with the weighted projective lines associated to $\cP_X, \cP_{X'}$  respectively. The right vertical arrow  induces an isomorphism $(f_{\cP})_\P$ and the left vertical arrow  induces  an isomorphism $f_{\P}$, so the claim follows.
\end{proof}

Recall the description of the autoequivalence group of $\cP_X$ from Lemma~\ref{lemma_AutCohP2}. Since the categories $\cP_X$ and $\coh\P_X$ are equivalent non-canonically, we will use invariant notation. Let $Q_{\cP_X}$ be the set of  $\tau$-orbits of simple objects in $\cP_X$  of length $2$, let $H_{\cP_X}$ be the set of simple objects in these  $\tau$-orbits. Then by Lemma~\ref{lemma_AutCohP2} we have 
\begin{lemma}
\label{lemma_AutCohP3}
There is a commutative diagram of groups where the rows are exact sequences:
\begin{equation}
\label{eq_AutCohP}
    \xymatrix{
\{e\}\ar[r] & \L_{\cP_X}\ar[r]\ar[d] &\Aut(\cP_X)\ar[r]\ar[d] &\Aut(\P_{\cP_X})\ar[r]\ar[d] &\{e\}\\
\{e\}\ar[r] & \Aut(H_{\cP_X}/Q_{\cP_X})\ar[r] &\Aut(H_{\cP_X},Q_{\cP_X})\ar[r] &\Aut(Q_{\cP_X})\ar[r] &\{e\}.
    }
\end{equation}
\end{lemma}

Recall that $Q_{\cP_X}$ is naturally identified with the set of weighted points on $\P_{\cP_X}$, and hence by Lemma~\ref{lemma_epsilon} with the set of weighted points of $\P_X$, which is $Q_X$. For $H_{\cP_X}$ we have the following
\begin{lemma}
    \label{lemma_HH}
    One has 
    $$H_{\cP_X}=\{\cO_X(-h), h\in H_X\}.$$
    In particular, $H_{\cP_X}$ is naturally identified with $H_X$.
\end{lemma}
\begin{proof}
    Recall that there are ten $0$-classes in $H_X$, divided into pairs, so that for any pair $h,h'$ one has $h+h'=-K_X$. We will show that the Serre functor $S_{\bA_X}$ on $\bA_X$ sends $\cO(-h)$ to $\cO(-h')[1]$, therefore $\tau(\cO(-h))\cong \cO(-h')$ and  $\tau^2(\cO(-h))\cong \cO(-h)$. It would follow that $\cO(-h)\in H_{\cP_X}$, and the statement would follow by comparing the cardinalities.

    We perform mutations in the semi-orthogonal decomposition $\Db(X)=\langle \bB, \langle\cO(-h)\rangle,\langle\cO\rangle\rangle$, see~\cite{BondalKapranov_Representable-functors-Serre-functors-mutations}.  First, we have 
    $$L_{\bB}(\cO(-h))\cong S_{\bA_X} (\cO(-h)).$$
    On the other hand, we have  
    \begin{multline*}
        L_{\bB}(\cO(-h))\cong L_{\langle\bB, \langle\cO\rangle\rangle}(R_{\langle\cO\rangle}(\cO(-h)))\cong 
     L_{\langle\bB, \langle\cO\rangle\rangle}(\cO(h)[-1])\cong \\ \cong S_{\Db(X)}(\cO(h)[-1])\cong \cO(h+K_X)[1]\cong \cO(-h')[1],
    \end{multline*}
    where the second isomorphism is by~\cite[Lemma 3.14]{ElaginSchneiderShinder} and the fourth is because the Serre functor on $\Db(X)$ is given by $-\otimes\cO(K_X)[2]$.
    The statement now follows.
\end{proof}

We observe 
%(but do not explain) 
that the natural maps $\pi_X\colon H_X\to Q_X$ and  $\pi_{\cP_X}\colon H_{\cP_X}\to Q_{\cP_X}$ are compatible with the identifications 
$H_X\cong H_{\cP_X}$ and $Q_X\cong Q_{\cP_X}$ described above.

\begin{lemma}
\label{lemma_cube}
    There is a natural homomorphism from diagram~\eqref{eq_AutX} to diagram~\eqref{eq_AutCohP}.
\end{lemma}
\begin{proof}
    We construct a homomorphism from the right square of~\eqref{eq_AutX} to the right square of~\eqref{eq_AutCohP}, then it will extend uniquely to a homomorphism of the left squares.

    The map from $\Aut(X)$ to $\Aut(\cP_X)$ is given by $f\mapsto f_\cP$. The map from $\Aut(\P_X)$ to $\Aut(\P_{\cP_X})$ comes from the natural identification $\epsilon_X\colon \P_X\xra{\sim}\P_{\cP_X}$ from Lemma~\ref{lemma_epsilon}. The maps  $\Aut(H_X,Q_X)\xra{\sim}\Aut(H_{\cP_X},Q_{\cP_X})$ and  $\Aut(Q_X)\xra{\sim}\Aut(Q_{\cP_X})$ come from the identifications 
    $H_X\cong H_{\cP_X}$ and $Q_X\cong Q_{\cP_X}$ which are compatible with $\pi_{X}$ and $\pi_{\cP_X}$.

    We get a cubic diagram, where
    the back face is the right square of~\eqref{eq_AutX} and the front face is the right square of~\eqref{eq_AutCohP}. Its  top face is commutative by Lemma~\ref{lemma_epsilon} and the left face is commutative by the definition of the identification $H_X\cong H_{\cP_X}$ from Lemma~\ref{lemma_HH}. Commutativity of the lower face is evident, and commutativity of the right face  follows automatically by diagram chase.
    $$\xymatrix{\Aut(X) \ar[rr] \ar[rd]\ar[dd]&& \Aut(\P_X) \ar[rd]\ar@{-}[d]&\\
    & \Aut(\cP_X)\ar[dd]\ar[rr] & \ar[d] & \Aut(\P_{\cP_X})\ar[dd] \\
    \Aut^{\ev}(H_X/Q_X)\ar@{-}[r]\ar[rd] & \ar[r] & \Aut(Q_X)\ar[rd] & \\
    & \Aut^{\ev}(H_{\cP_X}/Q_{\cP_X})\ar[rr] && \Aut(Q_{\cP_X})}$$
\end{proof}

We finish this section with the following observation, which will be crucial for the proof of our main result.
\begin{lemma}
    \label{lemma_main}
    For any del Pezzo surface $X$ of degree $4$, the natural group homomorphisms 
    $$\Aut(X)\to \Aut(\bA_X), \quad \Aut_\kk(X)\to \Aut_\kk(\bA_X)$$
    have left inverse homomorphisms.
\end{lemma}
\begin{proof}
    By Proposition~\ref{prop_BO} and Lemma~\ref{lemma_AutDbCohP}, one has 
    $$\Aut(\bA_X)\cong \Aut (\cP_X)\times \Z.$$
    The natural homomorphism $\Aut(X)\to \Aut(\bA_X)$ extends to  $\Aut(\cP_X)$, so it suffices to show that 
    the natural homomorphism 
    $\Aut(X)\to \Aut(\cP_X)$ has a left inverse.\\
    By Lemma~\ref{lemma_AutX}, there is an isomorphism 
    $$\alpha\colon \Aut(X) \xra{\sim}\Aut^{\ev}(H_X,Q_X)\times_{\Aut(Q_X)} \Aut(\P_X).$$ 
    By Lemma~\ref{lemma_AutCohP3}, there is a homomorphism 
    $$\beta\colon \Aut(\cP_X) \to \Aut(H_{\cP_X},Q_{\cP_X})\times_{\Aut(Q_{\cP_X})} \Aut(\P_{\cP_X}).$$
    By Lemma~\ref{lemma_cube}, $\alpha$ and $\beta$ fit into a commutative diagram 
    $$\xymatrix{\Aut(X)\ar[r]^-\alpha \ar[d] & \Aut^{\ev}(H_X,Q_X)\times_{\Aut(Q_X)} \Aut(\P_X)\ar[r]^-\gamma &
    \Aut(H_X,Q_X)\times_{\Aut(Q_X)} \Aut(\P_X)\ar[d]^{\sim}\\
    \Aut(\cP_X) \ar[rr]^-\beta && \Aut(H_{\cP_X},Q_{\cP_X})\times_{\Aut(Q_{\cP_X})} \Aut(\P_{\cP_X}).}$$
    Here, $\alpha$ and the homomorphism on the right are invertible, and $\gamma$ has a left inverse by Lemma~\ref{lemma_BDsplit}; therefore the homomorphism on the left also has a left inverse. 

    For $\kk$-linear automorphisms/autoequivalences, the proof is the same, see Lemma~\ref{lemma_AutCohP} and Remark~\ref{remark_AutXklinear}.
\end{proof}

\subsection{Separable functors}

Here we introduce some tools from category theory which we find useful for the formulation of our main result.
\begin{definition}[See~{\cite{Nastasescu-VdBergh-VOstaeyen}}]
\label{def_separable}
A functor $\Phi\colon \bC\to \bD$ is called \emph{separable} if the induced maps $\Phi_{X,Y}\colon \Hom_\bC(X,Y)\to \Hom_\bD(\Phi(X),\Phi(Y))$ have left inverse maps, natural in $X$ and $Y$. Explicitly, if there are given maps $\Psi_{X,Y}\colon \Hom_\bD(\Phi(X),\Phi(Y))\to \Hom_\bC(X,Y)$ for all $X,Y$ in $\bC$ such that
\begin{enumerate}[label=(s\arabic*)]
    \item  $\Psi_{X,Y}(\Phi_{X,Y}(f))=f$ for all $f\colon X\to Y$ in $\bC$,    \label{s1}
    \item for all  $a\colon X'\to X$, $b\colon Y\to Y'$ in $\bC$, the diagram commutes:
    $$\xymatrix{ \Hom_\bD(\Phi(X),\Phi(Y)) \ar[r]^-{\Psi_{X,Y}} \ar[d]^{\Phi(b)\circ - \circ\Phi(a)} & \Hom_\bC(X,Y) \ar[d]^{b\circ - \circ a} \\ \Hom_\bD(\Phi(X'),\Phi(Y')) \ar[r]^-{\Psi_{X',Y'}}  & \Hom_\bC(X',Y').}$$
       \label{s2}
\end{enumerate}
\end{definition}

\begin{definition}[See~{\cite{ArdizzoniMenini}}]
\label{def_hseparable}
A functor $\Phi\colon \bC\to \bD$ is called \emph{heavily separable} if it is separable and the maps $\Psi_{X,Y}\colon \Hom_\bD(\Phi(X),\Phi(Y))\to \Hom_\bC(X,Y)$ from Definition~\ref{def_separable} can be chosen such that 
\begin{enumerate}[label=(s\arabic*), start=3]
    \item  $\Psi_{X,Z}(vu)=\Psi_{Y,Z}(v)\circ \Psi_{X,Y}(u)$ for all $u\colon \Phi(X)\to \Phi(Y)$, $v\colon \Phi(Y)\to \Phi(Z)$  in $\bC$.
    \label{s3}
\end{enumerate}    
Note that condition \ref{s2} above follows from conditions \ref{s1} and \ref{s3}.
\end{definition}    

Recall that a category $\bC$ is a \emph{groupoid} if all morphisms in $\bC$ are invertible.
For functors between groupoids, we have the following
\begin{lemma}
\label{lemma_hseparable}
    Let $\Phi\colon \bC\to \bD$ be a functor between groupoids. Then the following conditions are equivalent:
    \begin{enumerate}[label=(\alph*)]
        \item $\Phi$ is heavily separable,
        \item $\Phi$ is injective on isomorphism classes, and for any $X\in\bC$ the group homomorphism $\Phi_{X,X}\colon \Aut_\bC(X)\to \Aut_{\bD}(\Phi(X))$ has a left inverse.
     \end{enumerate}
\end{lemma}
\begin{proof}
    (a) $\Longrightarrow$ (b). Let $\Phi$ be heavily separable. Assume that $\Phi(X)\cong \Phi(Y)$, then 
     $\Hom(X, Y)$ is non-empty by~\ref{s1}; hence, $X\cong Y$ by the definition of a groupoid. Therefore $\Phi$ is injective on isomorphism classes of objects. To split the homomorphism $\Phi_{X,X}$, one can take $\Psi_{X,X}$, which sends $1$ to $1$ by~\ref{s1} and respects products by~\ref{s3}.
     
    (b) $\Longrightarrow$ (a). We need to define $\Psi_{X,Y}\colon \Hom(\Phi(X),\Phi(Y))\to \Hom(X,Y)$ for all $X,Y\in\bC$. Fix an object in every isomorphism class in $\bC$. If $X\not\cong Y$, then $\Hom(\Phi(X),\Phi(Y))$ is empty by assumptions of (b) and there is nothing to do. If $X\cong Y$, let $X_0$ be the fixed object isomorphic to $X,Y$, let $\psi\colon \Aut (\Phi(X_0))\to \Aut(X_0)$ be a splitting of $\phi=\Phi_{X_0,X_0}$. Choose some isomorphisms $x\colon X_0\xra{\sim}X, y\colon X_0\xra{\sim}Y$, and define $\Psi_{X,Y}$ by identifying $\Hom(X,Y)$ with $\Aut(X_0)$ using $x,y$ and identifying $\Hom(\Phi(X),\Phi(Y))$ with $\Aut(\Phi(X_0))$ using $\Phi(x),\Phi(y)$:
    $$\Psi_{X,Y}(u)=y\psi(\Phi(y)^{-1}u\Phi(x))x^{-1}$$
    for all $u\in \Hom(\Phi(X), \Phi(Y))$. One checks property~\ref{s1}: for any $f\in\Hom(X,Y)$, one has
    \begin{multline*}
        \Psi_{X,Y}(\Phi_{X,Y}(f))=y\psi(\Phi(y)^{-1}\Phi(f)\Phi(x))x^{-1}=\\
    =y\psi(\Phi(y^{-1}fx))x^{-1}=y\psi(\phi(y^{-1}fx))x^{-1}=yy^{-1}fxx^{-1}=f.
    \end{multline*}
    One checks property~\ref{s3}: for $u\in\Hom(\Phi(X),\Phi(Y))$ and $v\in\Hom(\Phi(Y),\Phi(Z))$ one has (where $z\colon X_0\xra{\sim}Z$ is an isomorphism):
    \begin{multline*}
        \Psi_{Y,Z}(v)\Psi_{X,Y}(u)=z\psi(\Phi(z)^{-1}v\Phi(y))y^{-1}y\psi(\Phi(y)^{-1}u\Phi(x))x^{-1}=\\
        =z\psi(\Phi(z)^{-1}v\Phi(y))\psi(\Phi(y)^{-1}u\Phi(x))x^{-1}=
        z\psi(\Phi(z)^{-1}v\Phi(y)\Phi(y)^{-1}u\Phi(x))x^{-1}=\\
        =z\psi(\Phi(z)^{-1}vu\Phi(x))x^{-1}=\Psi_{X,Z}(vu).
    \end{multline*}
    Therefore, $\Phi$ is heavily separable, since~\ref{s2} follows automatically.
\end{proof}
\begin{remark}
\label{remark_Psicanonical}
    One can check that morphisms $\Psi_{X,Y}$, constructed in the proof of the (b) to (a) implication of Lemma~\ref{lemma_hseparable}, do not depend on the choice of isomorphisms $x,y$. Further, suppose that splitting homomorphisms $\psi_{X}\colon \Aut(\Phi(X))\to \Aut(X)$ from (b) are compatible in the following sense: for any isomorphism $a\colon X\xra{\sim}X'$, the diagram 
    $$\xymatrix{ \Aut(\Phi(X)) \ar[r]^-{\psi_{X}} \ar[d]^{\Phi(a)\circ - \circ\Phi(a)^{-1}} & \Aut(X) \ar[d]^{a\circ - \circ a^{-1}} \\ \Aut(\Phi(X')) \ar[r]^-{\psi_{X'}}  & \Aut(X')}$$
    commutes. Then the morphisms $\Psi_{X,Y}$, constructed in the proof, do not depend on the choice of objects $X_0$. 
\end{remark}

\begin{remark}
    For a functor $\Phi\colon \bC\to \bD$, where $\bC$ is a groupoid, the following  characterisation is easy to check. The functor $\Phi$ is heavily separable if and only if the induced functor $\bC\to \im \Phi$ has a left inverse functor, that is, a functor $\Psi\colon \im\Phi\to \bC$ such that $\Psi\Phi\cong \id_{\bC}$. Indeed, if such functor $\Psi$ exists, then the induced maps
    $$\Psi_{X,Y}\colon \Hom_\bD(\Phi(X),\Phi(Y))\xra{\Psi} \Hom_\bC(\Psi\Phi(X), \Psi\Phi(Y))\cong \Hom_\bC(X,Y)$$
    satisfy Definition~\ref{def_hseparable}. Conversely, suppose $\Phi$ is heavily separable and $\Psi_{X,Y}$ are as in Definition~\ref{def_hseparable}. One can choose a full subcategory $\bC_0\subset \bC$ such that $\Phi$ is a bijection on objects $\Ob\bC_0\xra{\sim}\Ob\im\Phi$. Then define $\Psi\colon \im\Phi\to \bC_0$ as the inverse to $\Phi$ on objects and given by  $\Psi_{X,Y}$ on morphisms for $X,Y\in \Ob\bC_0$. By construction, $\Psi\Phi=\id$ on $\bC_0$. It remains to note that $\bC_0$ is essential in $\bC$: for any $X\in \bC$, there is $X_0\in \bC$ such that $\Phi(X)=\Phi(X_0)$. It follows from~\ref{s1} that $\Hom(X, X_0)$ is non-empty, hence $X\cong X_0$ by the definition of a groupoid. It easily follows then that $\Psi\Phi$ is isomorphic to $\id$ on the whole of $\bC$.
\end{remark}

\section{Main results}
\label{section_main}

In this section we prove our main results. 

\subsection{Algebraically closed field case}
In this subsection, we assume that the base field $\kk$ is algebraically closed and of characteristic $\ne 2$. We claim that for any del Pezzo surfaces $X,Y$ over~$\kk$ of degree $4$ and any equivalence $\bA_X\xra{\sim}\bA_Y$, there is a canonically defined isomorphism $X\xra{\sim}Y$. Let $(\dP4, \Iso)$ and $(\dP4, \Iso_\kk)$ denote the categories where objects are del Pezzo   surfaces of degree $4$ over $\kk$ and the morphisms are isomorphisms  and $\kk$-linear isomorphisms of surfaces respectively. 
For essentially small $\kk$-linear triangulated categories $\bC,\bD$, denote by $\Iso(\bC,\bD)$ and 
$\Iso_\kk(\bC,\bD)$ the sets of isomorphism classes of exact equivalences  (resp. $\kk$-linear exact equivalences) $\bC\to\bD$.
Let $(\TrCat,\Iso)$ and $(\TrCat,\Iso_\kk)$ be the categories where objects are essentially small triangulated $\kk$-linear categories and the morphisms from $\bC$ to $\bD$ are given by 
$\Iso(\bC,\bD)$ and 
$\Iso_\kk(\bC,\bD)$ respectively.
\begin{theorem}
\label{th_main}
    The natural functors 
    $$\Phi \colon (\dP4, \Iso)\to (\TrCat, \Iso)\quad\text{and} \quad  (\dP4, \Iso_\kk)\to (\TrCat, \Iso_\kk)$$
    taking $X$ to $\bA_X$ and $f$ to $f_\bA$
    are heavily separable. Explicitly, for any surfaces $X,Y$ in $\dP4$ and equivalence $g\colon \bA_X\xra{\sim}\bA_Y$, there is a well-defined isomorphism $\Psi_{X,Y}(g)\colon X\xra{\sim}Y$ such that $\Psi_{X,Y}(\Phi_{X,Y}(f))=f$ for all $f\in \Iso(X,Y)$, and $\Psi_{Y,Z}(h)\Psi_{X,Y}(g)=\Psi_{X,Z}(hg)$ for all $g\in\Iso(\bA_X,\bA_Y)$, $h\in\Iso(\bA_Y,\bA_Z)$. If $g$ is $\kk$-linear, then $\Psi_{X,Y}(g)$ is also $\kk$-linear.
\end{theorem}
\begin{proof}
    We use Lemma~\ref{lemma_hseparable}. First, we observe that $\Phi$ is injective on isomorphism classes. Indeed, if $\bA_X\cong \bA_Y$, then $\cP_X\cong \cP_Y$ by Lemma~\ref{lemma_AutDbCohP}, hence $\P_X\cong \P_Y$ by Lemma~\ref{lemma_new}, and $X\cong Y$ by Lemma~\ref{lemma_diagonal}. Second, the group homomorphisms $\Aut(X)\to \Aut(\bA_X)$ have left inverse homomorphisms by Lemma~\ref{lemma_main}, and we conclude the proof using Lemma~\ref{lemma_hseparable}. For $\kk$-linear isomorphisms/equivalences the proof is the same.
\end{proof}

The same result holds in the $G$-equivariant setting. We fix a group $G$ and consider $G$-actions on del Pezzo surfaces and on triangulated categories. For the definitions of group actions on categories and equivariant functors, we refer to~\cite[Section 3.4]{ElaginSchneiderShinder}. 
Let $(\GdP4, \GIso)$ denote the category where objects are $G$-del Pezzo surfaces of degree $4$ over $\kk$, and the morphisms are $G$-equivariant isomorphisms of surfaces. We do not assume that the action is $\kk$-linear. Let $(\GTrCat,\GIso)$ be the category where objects are essentially small triangulated $\kk$-linear categories equipped with a $G$-action, and the morphisms are $G$-equivariant exact equivalences of categories up to  $G$-isomorphisms. We also consider non-full subcategories $(\GdP4, \GIso_\kk)\subset (\GdP4, \GIso)$ and $(\GTrCat,\GIso_\kk)\subset (\GTrCat,\GIso)$, where objects are the same, and morphisms are given by $\kk$-linear isomorphisms/equivalences.  
There are natural  functors
    $$(\GdP4, \GIso)\to (\GTrCat, \GIso), \qquad  (\GdP4, \GIso_\kk)\to (\GTrCat, \GIso_\kk).$$
Indeed, let $X$ be a $G$-del Pezzo surface of degree $4$, where the action is given by automorphisms $\rho_g, g\in G$, of $X$. There is  a $G$-action on $\Db(X)$ given by direct image functors; it  restricts to a $G$-action on $\bA_X$: an element $g$ acts by the autoequivalence $\Phi(\rho_g)=(\rho_g)_{\bA}$. Further, let $X,Y$ be $G$-del Pezzo surfaces of degree $4$, and $f\colon X\xra{\sim}Y$ be a $G$-equivariant isomorphism. Then the functor $f_*\colon \Db(X)\xra{\sim}\Db(Y)$ naturally becomes a $G$-equivariant functor and restricts to a $G$-equivariant functor $\Phi(f)=f_\bA\colon \bA_X\xra{\sim}\bA_Y$.
\begin{theorem}
\label{th_mainG}
    The  natural functors 
    $$\Phi\colon (\GdP4, \GIso)\to (\GTrCat, \GIso)\quad\text{and} \quad (\GdP4, \GIso_\kk)\to (\GTrCat, \GIso_\kk)$$
    taking $(X,(\rho_g))$ to $(\bA_X,((\rho_g)_{\bA}))$ and $f$ to $f_\bA$
    are heavily separable. In particular, if, for $G$-del Pezzo surfaces $X$ and $Y$ of degree $4$, the categories $\bA_X$ and $\bA_Y$ are $G$-equivariantly equivalent (resp. $\kk$-linearly $G$-equivariantly equivalent), then $X$ and $Y$ are $G$-equivariantly isomorphic (resp. $G$-equivariantly isomorphic over $\kk$).
\end{theorem}
\begin{proof}
    Everything follows from Theorem~\ref{th_main}.
    We will prove that the first functor is heavily separable; for the second one, the arguments are the same. Let $X,Y$ be $G$-del Pezzo surfaces of degree $4$. Let $u\colon \bA_X\xra{\sim}\bA_Y$ be a $G$-equivariant equivalence. Let $f=\Psi_{X,Y}(u)$ be the isomorphism $X\xra{\sim}Y$ provided by Theorem~\ref{th_main}. We have to prove that $f$ is $G$-equivariant, that is, for all $g\in G$, we have 
    $\rho_g^Yf=f\rho_g^X$. Indeed, we have $(\rho_g^Y)_\bA \circ u\cong u\circ (\rho_g^X)_\bA$ since $u$ is $G$-equivariant. Applying $\Psi_{X,Y}$ from Theorem~\ref{th_main}, we get the desired equality: 
    $$\Psi_{X,Y}((\rho_g^Y)_\bA \circ u)=\Psi_{Y,Y}((\rho_g^Y)_\bA)\Psi_{X,Y}(u)=\Psi_{Y,Y}(\Phi_{Y,Y}(\rho_g^Y))f=\rho_g^Y f,$$
    and similarly, $\Psi_{X,Y}(u\circ (\rho_g^X)_\bA)=f\rho_g^X$. Therefore, the maps $\Psi_{X,Y}$ from Theorem~\ref{th_main} take $G$-equivariant equivalences to $G$-equivariant isomorphisms, and satisfy the properties from Definition~\ref{def_hseparable} again by Theorem~\ref{th_main}.
\end{proof}

\subsection{Arbitrary perfect field case}

Now we assume that $\kk$ is any perfect field and $\bar\kk$ is its algebraic closure. The field extension $\kk\subset \bar\kk$ is a (possibly infinite) Galois extension; let $G$ be its Galois group. We suppose that $\har \kk\ne 2$.
For a del Pezzo surface $X$ over $\kk$, we denote by $\oX$ the scalar extension $X\times_{\kk}\Spec\bar\kk$, which is a del Pezzo surface over $\bar\kk$. As for the algebraically closed case, the orthogonal subcategory $\bA_X=\cO_X^\perp \subset \Db(X)$ is defined, and, as before, we prove that $X$ is uniquely determined by this subcategory. 

Let us say that an equivalence $\bA_X\xra{\sim}\bA_Y$ is \emph{of Fourier--Mukai type} if
the composition
$$\Db(X)\to \bA_X\xra{\Phi}\bA_Y\to \Db(Y)$$
is given by a Fourier--Mukai kernel $\cE\in\Db(X\times Y)$ (where the first and the last arrows are the projection and the inclusion functors, respectively). In the statement below, we restrict ourselves to equivalences of Fourier--Mukai type.

\begin{theorem}
\label{theorem_main-perfect}
    If $X,Y$ are del Pezzo surfaces of degree $4$ over a perfect field $\kk$ with $\har\kk\ne 2$, and there is a $\kk$-linear equivalence $\bA_X\xra{\sim}\bA_Y$ of Fourier--Mukai type, then $X$ is $\kk$-linearly isomorphic to $Y$. 
\end{theorem}
\begin{remark}
We believe that an analogue of Theorem~\ref{th_main} holds for any perfect field: the natural functor  $(\dP4, \Iso_\kk)\to (\TrCat, \Iso_\kk)$ is heavily separable; that is,   for any $\kk$-linear equivalence $\bA_X\xra{\sim}\bA_Y$ there is a well-defined isomorphism $X\xra{\sim}Y$ over $\kk$. Its proof consists of three steps: (1) given an equivalence $u\colon \bA_X\xra{\sim}\bA_Y$, construct a $G$-equivariant equivalence $\bA_{\oX}\xra{\sim}\bA_{\oY}$, (2)
construct a $G$-equivariant isomorphism  $\oX\to\oY$, (3) obtain an isomorphism $X\xra{\sim}Y$. Step (2) is by Theorem~\ref{th_mainG}, while step (3) is standard Galois descent.
However, doing step (1)
would require knowing that $u$ is given by a unique Fourier--Mukai kernel or lifts properly to a dg-enhancement. We prefer not to overload the paper with the relevant details and to prove a weaker statement as in Theorem~\ref{theorem_main-perfect}.
\end{remark}
\begin{proof}[Proof of Theorem~\ref{theorem_main-perfect}]
    We deduce the statement by passing to the algebraic closure and using Theorem~\ref{th_mainG}. Note that $\oX$  and $\oY$ carry natural $G$-actions and may be viewed as $G$-surfaces. We refer to~\cite[Section 3.5]{ElaginSchneiderShinder} for the Galois descent for subcategories in derived categories of coherent sheaves. 
    
    First, we claim that the orthogonal subcategories $\bA_{\oX}$ and $\bA_{\oY}$ for $\oX$ and $\oY$ are $\bar\kk$-linearly $G$-equivariantly equivalent. To explain this, assume that a $\kk$-linear equivalence $\bA_X\to\bA_Y$ is of Fourier--Mukai type and is given by a kernel $\cE$. Consider the scalar extension $\bar{\cE}\in\Db(\oX\times \oY)$ of~$\cE$. By~\cite[Section 6]{Kuznetsov_BaseChange}, the kernel $\bar{\cE}$ produces an equivalence $\bA_{\oX}\xra{\sim}\bA_{\oY}$, 
    which is clearly $\bar\kk$-linear and  $G$-equivariant.
    We use now Theorem~\ref{th_mainG} to obtain a  $\bar\kk$-linear $G$-equivariant isomorphism $f\colon \oX\xra{\sim}\oY$. By Galois descent, $f$ produces a $\kk$-linear isomorphism $X\xra{\sim}Y$.
\end{proof}

\subsection{An application to birationality}

Using \emph{atomic theory} introduced in~\cite{ElaginSchneiderShinder} and the main result of this paper, we explain that two minimal del Pezzo surfaces of degree $4$ are birational if and only if they are isomorphic. There are two versions of the statement: for minimal $G$-del Pezzo surfaces over an algebraically closed field, and for  minimal del Pezzo surfaces over any perfect  field, the latter being a special case of the former.

For an abelian group $M$ with an action of a group $G$, we denote by $M^G\subset M$ the subgroup of invariants. 
A $G$-del Pezzo surface $X$ is said to be \emph{$G$-minimal} if $\rank(\Pic(X)^G)=1$.

\begin{theorem}\label{th_birationalG}
Let $\kk$ be an algebraically closed field of characteristic $\ne 2$, and let $X,Y$ be del Pezzo surfaces of degree $4$ over $\kk$. Suppose a group $G$ acts on $X,Y$ (not necessarily $\kk$-linearly) so that $X,Y$ are $G$-minimal.
Then the following conditions are equivalent:
\begin{enumerate}[label=(\alph*)]
    \item\label{it:dP4--iso} surfaces $X$ and $Y$ are $G$-isomorphic;
    \item\label{it:dP4--bir} surfaces $X$ and $Y$ are $G$-birational;
    \item\label{it:dP4--atoms} categories $\bA_X$ and $\bA_Y$ are $G$-equivariantly equivalent.
\end{enumerate}
\end{theorem}
\begin{proof}
Implication \ref{it:dP4--iso} $\implies$ \ref{it:dP4--bir} is trivial. Implication \ref{it:dP4--bir} $\implies$ \ref{it:dP4--atoms} follows from~\cite[Prop. 5.1]{ElaginSchneiderShinder} and Definition~\cite[Def. 4.21]{ElaginSchneiderShinder} of atoms for a $G$-minimal del Pezzo surface of degree $4$. Implication \ref{it:dP4--atoms} $\implies$ \ref{it:dP4--iso} is by Theorem~\ref{th_mainG}.
\end{proof}

\begin{theorem}\label{th_birational-perfect}
Let $\kk$ be a perfect field of characteristic $\ne 2$, and let $X,Y$ be minimal del Pezzo surfaces of degree $4$ over $\kk$. 
Then the following conditions are equivalent:
\begin{enumerate}[label=(\alph*)]
    \item surfaces $X$ and $Y$ are  isomorphic over $\kk$;
    \item surfaces $X$ and $Y$ are  birational over $\kk$;
    \item  categories $\bA_X$ and $\bA_Y$ are $\kk$-linearly equivalent by an equivalence of Fourier--Mukai type.
\end{enumerate}
\end{theorem}
\begin{proof}
The same as for Theorem~\ref{th_birationalG}. 
Implication \ref{it:dP4--bir} $\implies$ \ref{it:dP4--atoms} follows from~\cite{ElaginSchneiderShinder}. Note that \emph{loc.cit.} proves that the categories $\bA_X$ and $\bA_Y$ are $\kk$-linearly equivalent, but it is not checked that they are equivalent by a functor of Fourier--Mukai type. However, the constructed equivalence comes from pull-back functors, mutation functors, and projection functors, which are of Fourier--Mukai type.
Implication \ref{it:dP4--atoms} $\implies$ \ref{it:dP4--iso} is by Theorem~\ref{theorem_main-perfect}.
\end{proof}

\section{Indecomposability of the orthogonal subcategory}
\label{section_indecomposability}

In the final section we establish that the orthogonal component on \emph{minimal} del Pezzo surfaces of degree $4$ is semi-orthogonally indecomposable; an analogous statement holds for minimal conic bundles of degree $\le 4$.

Recall that a $G$-del Pezzo surface $X$ is called \emph{$G$-minimal} if $\rank(\Pic(X)^G) = 1$. 

\begin{theorem}
\label{th_indecomposability}
Let $\kk$ be an algebraically closed field of characteristic $\ne 2$, and let $X$ be a del Pezzo surface of degree $4$ over $\kk$. Suppose a group $G$ acts on $X$ (not necessarily $\kk$-linearly) in such a way that $X$ is $G$-minimal. Then the orthogonal subcategory $\bA_X$ admits no semi-orthogonal decompositions into $G$-invariant subcategories.
\end{theorem}
\begin{corollary}
\label{corollary_indecomposability}
Let $\kk$ be a perfect field of characteristic $\ne 2$, and let $X$ be a minimal del Pezzo surface of degree $4$ over $\kk$. Then the orthogonal subcategory $\bA_X$ admits no semi-orthogonal decompositions.
\end{corollary}

\begin{proof}[Proof of Corollary~\ref{corollary_indecomposability}]
    Consider the scalar extension $\oX = X \times_\kk \Spec\bar\kk$, endowed with the natural action of the Galois group $G = \Gal(\bar\kk/\kk)$. The surface $\oX$ is $G$-minimal, hence Theorem~\ref{th_indecomposability} is applicable to $\oX$. 

    The claim then follows from the behaviour of semi-orthogonal decompositions under scalar extension: any semi-orthogonal decomposition of $\bA_X$ induces, by scalar extension, a $G$-invariant semi-orthogonal  decomposition of $\bA_{\oX}$; see, for instance, \cite[Section~3.5]{ElaginSchneiderShinder}.
\end{proof}

\begin{proof}[Proof of Theorem~\ref{th_indecomposability}]
Since $X$ is rational, there exists a $G$-equivariant isomorphism (see, for instance, \cite[Lemma 4.2]{KarmazynKuznetsovShinder})
$$K_0(X)\xra{(\rank, c_1,\chi)} \Z \oplus \Pic(X) \oplus \Z,$$
where the group $G$ acts trivially on both copies of $\Z$. Consequently, we obtain
$$\rank((K_0(X)^G)) = 3.$$
In the semi-orthogonal decomposition
$$\Db(X) = \langle \bA_X, \langle \cO_X \rangle \rangle$$
the structure sheaf $\cO_X$ is $G$-invariant and exceptional. Hence, passing to $G$-invariants on $K$-theory, we deduce
$$\rank((K_0(\bA_X)^G)) = 2.$$
Recall that there is an equivalence of triangulated categories
$$ \bA_X \cong \Db(\coh \P), $$
where $\P = \P_X$ is the weighted projective line of type $(2,2,2,2,2)$ associated with $X$. The group $G$ acts on the abelian category $\coh \P$ in such a way that the above equivalence is $G$-equivariant. Therefore,
$$ \rank((K_0(\coh \P)^G)) = \rank((K_0(\bA_X)^G)) = 2. $$
Thus, the proof is reduced to a statement concerning weighted projective lines; see Proposition~\ref{prop_indecomp-for-P} below.
\end{proof}
In the following proposition, we work in a more general setting by considering not only weighted projective lines, but also weighted projective curves of arbitrary genus, subject to only mild constraints on the weights. This generalization is motivated by subsequent applications; see Theorem~\ref{th_indecomposability2}. Let $\P$ denote a weighted projective curve (see, for instance,~\cite{Elagin_WPC}). 

We write $K_0(\P)=K_0(\coh\P)$ for the Grothendieck group of the abelian category $\coh \P$, and denote by $K_0(\P)_{\numerical}$ the numerical Grothendieck group, defined as the quotient of $K_0(\P)_\Q$ by the kernel of the Euler form. Observe that if $\P$ is rational, then the Euler form is non-degenerate, and consequently $K_0(\P)_{\numerical}=K_0(\P)_\Q$.

\begin{prop}
    \label{prop_indecomp-for-P}
Assume that the base field $\kk$ is algebraically closed.
Let $\P$ be a weighted projective curve over $\kk$ and, in the case where the underlying curve is rational, suppose in addition that $\P$ is not of domestic type. Furthermore, let a group $G$ act on the abelian category $\coh \P$ of coherent sheaves on $\P$ in such a way that $\rank(K_0(\P)_{\numerical}^G)=2$.
Under these assumptions, the bounded derived category $\Db(\coh \P)$ does not admit semi-orthogonal decompositions into $G$-invariant subcategories.
\end{prop}
\begin{proof}
Let $F^1K_0(\P)_{\numerical}\subset K_0(\P)_{\numerical}$ be the subspace generated by classes of torsion sheaves; then there is a $G$-invariant exact sequence $0\to F^1K_0(\P)_{\numerical}\to K_0(\P)_{\numerical}\xra{\rank} \Q\to 0$. It follows that 
\begin{equation}
    \label{eq_rank}
    \rank(F^1K_0( \P)_{\numerical}^G)=1.
\end{equation}
The proof of the Proposition is based on the analysis of triangulated  subcategories in $\Db(\coh\P)$ from~\cite{Elagin_WPC}. Recall the necessary results. For a thick triangulated subcategory $\bA\subset \Db(\coh \P)$ let $\cA=\bA\cap \coh \P$; this is an exact abelian subcategory of $\coh \P$ and the natural functor $\Db(\cA)\to \Db(\coh \P)$ is fully faithful with  image $\bA$, see, for instance,~\cite[Prop. 3.7]{Elagin_WPC}. 
Such a triangulated subcategory $\bA$ is called \emph{quiver-like} if the abelian category $\cA$ is equivalent to the category $\moddo \kk\Gamma$ of nilpotent finite-dimensional representations of some quiver $\Gamma$. In this case the simple objects of $\cA$ correspond to simple modules concentrated at vertices of $\Gamma$. If $\bA$ is left or right admissible then $\Gamma$ has to be finite and without oriented cycles (\cite[Prop. 4.6]{Elagin_WPC}); in particular, the simple objects of $\cA$ form a full  exceptional collection in $\bA$.

Assume that a semi-orthogonal decomposition $\Db(\coh\P)=\langle\bA,\bB\rangle$ into non-trivial subcategories is given. Then one of the following holds by~\cite[Th. 8.2, Def. 7.1]{Elagin_WPC}:
\begin{enumerate}[label=(\alph*)]
    \item both $\bA$ and $\bB$ are quiver-like; \label{eq_case-a}
    \item either $\bA$ or $\bB$ is quiver-like and contains only torsion sheaves. \label{eq_case-b}
\end{enumerate}
Assume case~\ref{eq_case-a}. Let $E_1,\ldots, E_n$ and $F_1,\ldots,F_m$ be full exceptional collections in $\bA, \bB$ formed by simple objects in $\cA, \cB$ respectively. Then $\Db(\P)$ has a full exceptional collection, therefore
$K_0(\P)$ is finitely generated, and $\P$ must be rational, so we have $K_0(\P)_{\numerical}=K_0(\P)_\Q$.  
Since $G$ preserves the abelian category $\cA$, $G$ must permute its simple objects $E_i$, and similarly for $F_j$. It is easy to see that the collections $E_1,\ldots, E_n$ and $F_1,\ldots,F_m$ split into exceptional orthogonal blocks (for instance, see~\cite[Lemma 2.12]{Ballard-Duncan-McFaddin}), each being a $G$-orbit of sheaves. For any such $G$-orbit 
$\{E_i\}_{i\in I}$ the class $\oplus_{i\in I}[E_i]\in K_0( \P)$ is $G$-invariant, and such classes are linearly independent in $K_0(\P)$ because the classes $[E_1],\ldots,[E_n], [F_1], \ldots,[F_m]$ of objects of a full exceptional collection form a basis in $K_0(\P)$. By the assumption $\rank(K_0(\P)^G)=2$ we get that there are just two such orbits, that is, $G$ acts transitively on the orthogonal blocks $E_1,\ldots, E_n$ and $F_1,\ldots,F_m$. By Lemma~\ref{lemma_EF} below,  the graded vector space $\Hom^\bul(\oplus_i E_i, \oplus_j F_j)$ is concentrated in one degree $d=0$ or $1$. Then the collection $E_1,\ldots, E_n, F_1[d],\ldots,F_m[d]$ is a full strong exceptional collection in $\Db(\coh\P)$, consisting of two blocks. It provides (by~\cite{Bondal_associative}) an equivalence $\Db(\coh\P)\cong \Db(\modd A)$, where $A=\End((\oplus_i E_i)\oplus (\oplus_j F_j[d]))$ is the path algebra of a bipartite quiver, hence $A$ is hereditary. This gives a contradiction because $\P$ is a weighted projective line not of domestic type, see, for instance,~\cite[Section 10.1]{Lenzing_hercat}.

Assume now case~\ref{eq_case-b}; suppose without loss of generality that $\bA$ is quiver-like and contains only torsion sheaves. Let $E_1,\ldots, E_n$ be a full exceptional collection in $\bA$ formed by simple objects in~$\cA$. As above, we observe that this collection splits into orthogonal blocks which are $G$-orbits. Let $\{E_i\}_{i\in I}$ be such a block. Consider the classes $e=\oplus_{i\in I}[E_i]$ and $f=[\cO_x]$ in $K_0(\P)_{\numerical}$, where $x\in \P$ is any non-weighted point. These classes are $G$-invariant; also, they are not proportional. Indeed, for the Euler form one has $\chi(e,e)=\sum_{i\in I}\chi(E_i,E_i)=|I|>0$ since $\{E_i\}_{i\in I}$ is an orthogonal exceptional block, but $\chi(f,f)=0$. Note that the $E_i$'s are torsion sheaves, so both $e,f\in F^1K_0(\P)_{\numerical}^G$, and we get  a contradiction with~\eqref{eq_rank}.

Therefore both cases~\ref{eq_case-a} and~\ref{eq_case-b} are impossible and $\Db(\P)$ has no semi-orthogonal decompositions into $G$-invariant subcategories.
\end{proof}

Using Proposition~\ref{prop_indecomp-for-P} we can also show indecomposability of the ``big atom'' of a minimal two-dimensional conic bundle. Let $\kk$ be algebraically closed. Assume $f\colon X\to C$ is a conic bundle over  a smooth projective curve $C$, and assume a group $G$ acts on $X$ and $C$ compatibly. There is a $G$-invariant semi-orthogonal decomposition~\cite[Prop. 2.3]{Bridgeland_Flops}
\begin{equation}
\label{eq_SOD-conic-bundle}
    \Db(X)=\langle \ker f_*, f^*\Db(C)\rangle, 
\end{equation}
where $\ker f_*\subset \Db(X)$ is the full subcategory of objects sent to zero by the derived push-forward functor $f_*$ and $f^*\Db(C)$ is the image of the fully faithful derived pull-back functor $f^*$.
It is said that $X/C$ is \emph{relatively $G$-minimal} if $\rank(\Pic(X/C)^G)=1$; suppose this is the case. Assume that~$C$ is irrational or $C$ is rational and $K_X^2\le 4$. Then in~\cite[Defs. 2.11, 4.21, 5.11, Prop. 2.12]{ElaginSchneiderShinder} the category 
$$\bA_{X/C}=\ker f_*\subset \Db(X)$$
is called the \emph{big atom} of $X/C$. Now we justify this name by showing that $\bA_{X/C}$ is indeed indecomposable.

\begin{theorem}
\label{th_indecomposability2}
Let $\kk$ be an algebraically closed field of characteristic $\ne 2$, let $f\colon X\to C$ be a conic bundle over a smooth curve. Suppose  that $C$ is irrational or $C$ is rational and $K_X^2\le 4$. Assume a group $G$ acts on $X$ and $C$ compatibly (not necessarily $\kk$-linearly) so that $X/C$ is relatively $G$-minimal. Then the big atom  $\bA_{X/C}$ has no semi-orthogonal decompositions into $G$-invariant subcategories.
\end{theorem}
\begin{corollary}
\label{corollary_indecomposability2}
Let $\kk$ be a perfect field of characteristic $\ne 2$, let $f\colon X\to C$ be a relatively minimal conic bundle over a smooth curve. Suppose  that $C$ is geometrically irrational or $C$ is geometrically rational and $K_X^2\le 4$.  Then the big atom  $\bA_{X/C}$ has no semi-orthogonal decompositions.
\end{corollary}
\begin{proof}[Proof of Corollary~\ref{corollary_indecomposability2}]
  See the proof of Corollary~\ref{corollary_indecomposability}.
\end{proof}
\begin{proof}[Proof of Theorem~\ref{th_indecomposability2}]
As in the proof of Theorem~\ref{th_indecomposability} we get a constraint on the rank of the $G$-invariant part of the Grothendieck group of  $\bA_{X/C}$. To accommodate irrational base curves, we need to consider the \emph{numeric} version of the relevant invariants. 
Let $K_0(\bT)_{\numerical}=(K_0(\bT)_\Q)/\ker\chi$ be the numeric Grothendieck group of a triangulated category $\bT$ (with respect to the Euler form $\chi$). For $0\le k\le 2$, let $N^k(X)_{\numerical}=(A^k(X)_\Q)/_{\equiv}$ be the $k$-th numeric Chow group of $X$. Denote by $\Pic(X)_{\numerical}=(\Pic (X)_\Q)/\equiv$ the numeric Picard group of $X$ (with respect to the intersection form).  Then the Chern character gives an isomorphism (for instance, by~\cite[Lemma 2.1]{ElaginLunts2016_Exceptional-Collections-on-del-Pezzo})
$$\mathrm{Ch}\colon K_0(X)_{\numerical}\xra{\sim}N^0(X)_{\numerical}\oplus N^1(X)_{\numerical}\oplus N^2(X)_{\numerical}\cong \Q\oplus \Pic(X)_{\numerical}\oplus\Q.$$
This map is $G$-equivariant and the action on both summands $\Q$ is trivial. Therefore we have $\rank(K_0(X)_{\numerical}^G)=\rank(\Pic(X)_{\numerical}^G)+2$. For the projective curve $C$ we have: 
$K_0(C)_{\numerical}\cong \Q^2$ (rank and degree form a dual basis in this space) and $\Pic(C)_{\numerical}:=\Pic(C)_\Q/\Pic_0(C)_\Q\cong \Q$ (with degree forming  a dual basis). The group $G$ acts trivially on these spaces, so we get $\rank(K_0(C)_{\numerical}^G)=2$, $\rank(\Pic(C)_{\numerical}^G)=1$.
Note that $\Pic_0(X)= \Pic_0(C)$, so there is an exact sequence
$$0\to \Pic(C)_{\numerical}^G\to \Pic(X)_{\numerical}^G\to  \Pic(X/C)_{\Q}^G\to 0,$$
from which we get $\rank(\Pic(X)_{\numerical}^G)=1+1=2$. Hence $\rank(K_0(X)_{\numerical}^G)=2+2=4$.
The semi-orthogonal decomposition~\eqref{eq_SOD-conic-bundle} yields an isomorphism 
$K_0(X)_{\numerical}^G\cong K_0(\bA_{X/C})_{\numerical}^G\oplus K_0(C)_{\numerical}^G$,
and we conclude that $ \rank(K_0(\bA_{X/C})_{\numerical}^G)=2$.

Let $C_1\subset C$ be the finite set of closed points parametrising degenerate conics (i.e. reducible fibres of $f$). Then $X$ is the blow-up of $n=|C_1|$ points in fibres of a $\P^1$-bundle over $C$. If $C$ is rational, it follows from the assumption $K_X^2\le 4$ that $n\ge 4$. 

Consider the weighted projective curve $\P$ obtained from $C$ by putting weights $2$ at points of $C_1$. Then by~\cite[Th. 4.2, Cor. 3.16]{Kuznetsov2008_Derived-categories-of-quadric-fibrations} there is an equivalence 
$$\bA_{X/C}\cong \Db(\coh \P).$$
Moreover, by construction of \emph{loc.\,cit.} the $G$-action on $\bA_{X/C}$ comes from a $G$-action on $\coh\P$. 
We have $\rank(K_0(\P)_{\numerical}^G)=\rank(K_0(\bA_{X/C})_{\numerical}^G)=2$.

We can apply Proposition~\ref{prop_indecomp-for-P} to finish the proof. Note that our assumptions exclude  the case of a domestic weighted projective line: if $C$ is rational then the number $n$ of weighted points is~$\ge 4$.
\end{proof}

Finally we prove the lemma used in the proof of Proposition~\ref{prop_indecomp-for-P}.
Variants of this lemma are standard; we recall the proof in the form suitable for our situation.
\begin{lemma}
\label{lemma_EF}
    Let $\P$ be a weighted projective line and $(E_1,\ldots, E_n,F_1,\ldots,F_m)$ be an exceptional collection of coherent sheaves on $\P$, where $(E_1,\ldots, E_n)$ and $(F_1,\ldots,F_m)$ are orthogonal blocks. Assume a group of autoequivalences of $\coh\P$ is given such that $\{E_1,\ldots, E_n\}$ and $\{F_1,\ldots,F_m\}$ are orbits. Then either $\Hom(E_i,F_j)=0$ for all $i,j$ or $\Ext^1(E_i,F_j)=0$ for all $i,j$.
\end{lemma}
\begin{proof}
    Assume that $\Ext^1(E_i,F_j)\ne 0$ for some $i,j$; without loss of generality we may suppose that $i=j=1$. Consider the left mutation $\tF_1$ of $F_1$ through $(E_1,\ldots, E_n)$ (see \cite{BondalKapranov_Representable-functors-Serre-functors-mutations}), defined  by the exact triangle
    $$\oplus_i(\Hom^\bullet(E_i,F_1)\otimes E_i)\to F_1\to \tF_1\to.$$ Since the object $\tF_1$ is indecomposable, it has non-zero cohomology  exactly in one degree. The long exact sequence of cohomology sheaves associated with the above triangle is
    $$0\to \rH^{-1}(\tF_1)\to \oplus (\Hom(E_i,F_1)\otimes E_i)\to F_1\to \rH^0(\tF_1)\to \oplus (\Ext^1(E_i,F_1)\otimes E_i)\to 0.$$
    Since $\Ext^1(E_1,F_1)\ne 0$ it follows that $\rH^0(\tF_1)\ne 0$ and hence $\rH^{-1}(\tF_1)=0$. Therefore the evaluation  map $\oplus (\Hom(E_i,F_1)\otimes E_i)\to F_1$ is injective and hence any non-zero map $E_i\to F_1$ is also injective.

    From the group symmetry it follows that any non-zero map $E_i\to F_j$ is injective for all $i,j$. Similarly one shows, using a right mutation, that any non-zero map $E_i\to F_j$ is surjective for all $i,j$. Hence any non-zero map $E_i\to F_j$ must be an isomorphism and therefore cannot exist.    
\end{proof}

\bibliography{biblio_main}
\bibliographystyle{amsalpha}

\end{document}